\documentclass[a4paper,11pt,reqno]{amsart}
\usepackage[centering,includeheadfoot,margin=2.5 cm]{geometry}
\usepackage[percent]{overpic}
\usepackage{todonotes}

\usepackage{graphics,enumitem,epsfig,textcomp}
\usepackage{amsfonts,amsmath,amssymb,euscript,color,mathrsfs}
\usepackage[utf8]{inputenc}	
\usepackage{float} 
\newcommand{\clarissa}[1]{\textcolor{magenta}{\textbf{#1}}}
\usepackage[caption=false,subrefformat=parens,labelformat=parens]{subfig}
\usepackage{url}
\usepackage[symbol]{footmisc}
\usepackage{hyperref}
\usepackage{cases}
\usepackage{bigints}
\hypersetup{
    colorlinks=true,
    linkcolor=blue,
    filecolor=blue,      
    urlcolor=blue,
    citecolor=blue,
    }

\usepackage[square,numbers]{natbib}
\usepackage{amsmath,amsthm,amssymb,amsfonts,lineno}

\numberwithin{equation}{section}
\newtheorem{theorem}{Theorem}%[section]

\newtheorem{proposition}{Proposition}%[section]
\newtheorem{remark}{Remark}%[section]

\def\diver{\mathrm{div}}

\def\d{\,\mathrm{d}}
\def\eps{\varepsilon}
\def\R{\mathbb{R}}
\def\C{\hbox{\rlap{\kern.24em\raise.1ex\hbox
      {\vrule height1.3ex width.9pt}}C}}
\def\P{\hbox{\rlap{I}\kern.16em P}}
\def\Q{\hbox{\rlap{\kern.24em\raise.1ex\hbox
      {\vrule height1.3ex width.9pt}}Q}}
\def\M{\hbox{\rlap{I}\kern.16em\rlap{I}M}}
\def\Z{\hbox{\rlap{Z}\kern.20em Z}}
\def\({\begin{eqnarray}}
\def\){\end{eqnarray}}
\def\[{\begin{eqnarray*}}
\def\]{\end{eqnarray*}}
\def\part#1#2{\frac{\partial #1}{\partial #2}}

\def\grad{\nabla}
\def\pmb#1{\setbox0=\hbox{$#1$}
  \kern-.025em\copy0\kern-\wd0
  \kern-.05em\copy0\kern-\wd0
  \kern-.025em\raise.0433em\box0 }
\def\bar{\overline}

\def\d{\,\mathrm{d}}

\def\R{\mathbb{R}}

\def\epsilon{\varepsilon}

\def\P{\mathbb{P}}
\def\Q{\mathbb{Q}}

\newcommand{\dx}{\mathrm{d}x}

\def\cJH#1{\textcolor{blue}{\bf [#1]}}

\usepackage{cancel}
\title[Self-regulated biological transportation structures in 1D]{Self-regulated biological transportation structures with general entropy dissipations, part I: the 1D case}
\begin{document}
\pagenumbering{gobble}
\maketitle
\pagenumbering{arabic}

\centerline{
     {\large Clarissa Astuto}\footnote{Mathematical and Computer Sciences and Engineering Division,
         King Abdullah University of Science and Technology,
         Thuwal 23955-6900, Kingdom of Saudi Arabia;
         {\it clarissa.astuto@kaust.edu.sa}} \qquad
     {\large Jan Haskovec}\footnote{Mathematical and Computer Sciences
            and Engineering Division,
         King Abdullah University of Science and Technology,
         Thuwal 23955-6900, Kingdom of Saudi Arabia;
         {\it jan.haskovec@kaust.edu.sa}}\qquad
     {\large Peter Markowich}\footnote{Mathematical and Computer Sciences and Engineering Division,
         King Abdullah University of Science and Technology,
         Thuwal 23955-6900, Kingdom of Saudi Arabia;
         {\it peter.markowich@kaust.edu.sa}, and
         Faculty of Mathematics, University of Vienna,
        Oskar-Morgenstern-Platz 1, 1090 Vienna;
         {\it peter.markowich@univie.ac.at}}\qquad
         {\large Simone Portaro}\footnote{Mathematical and Computer Sciences
            and Engineering Division,
         King Abdullah University of Science and Technology,
         Thuwal 23955-6900, Kingdom of Saudi Arabia;
         {\it simone.portaro@kaust.edu.sa}}
     }
%\begin{big}
\medskip \medskip
% riga bianca da mettere
%\end{big}

\section*{abstract}
We study self-regulating processes modeling biological transportation networks as presented in \cite{portaro2023}. In particular, we focus on the 1D setting for Dirichlet and Neumann boundary conditions. We prove an existence and uniqueness result under the assumption of positivity of the diffusivity $D$. We explore systematically various scenarios and gain insights into the behavior of $D$ and its impact on the studied system. This involves analyzing the system with a signed measure distribution of sources and sinks. Finally, we perform several numerical tests in which the solution $D$ touches zero, confirming the previous hints of local existence in particular cases.  
 
\section{Introduction}
In recent years, there has been growing interest in studying principles and mechanisms of formation of biological transportation structures.
This line of research has numerous applications beyond biology, in particular in fields such as medicine, chemistry and engineering. 
This paper contributes to the discussion by studying the spatially one--dimensional version of self-regulated processes modeling biological transportation structures.
Such structures, in particular organization of leaf venation networks, vascular and neural networks, have been in the focus of biophysical and mathematical research communities, see, e.g., the recent works \cite{astuto2022comparison, astuto2023asymmetry, astuto2023finite, morel, couder, Corson, magnasco}. One is typically interested in geometric, topological and statistical properties of the optimal transportation structures, where the optimality criteria typically involve a trade-off between cost and transportation efficiency \cite{Barthelemy}. Biological transportation structures develop without centralized control \cite{Tero} and therefore they can be considered as emergent structures resulting from self-regulating processes.

In the recent work \cite{portaro2023} a class of self-regulating processes has been introduced,
which is governed by minimization of the entropy dissipation
\(  
    \label{eq:energy_functional1}
    E[D] = \int_{\Omega} \Phi''(u) \grad u \cdot D \grad u \dx.
\)
The functional $E=E[D]$ is to be minimized with respect to the symmetric and positive semidefinite diffusivity tensor field $D = D(x)$.
Here $\Phi : \R \rightarrow \R$ is the convex entropy generating function and $u = u(x)$ denotes the concentration of the transported medium, e.g., chemical species, ions or nutrients.
Moreover, $\Omega \subset \R^d$, $d \ge 1$ is a domain with $C^1$ boundary.

The functional \eqref{eq:energy_functional1} is coupled to the mass conservation equation
\(
    \label{eq:conservation_law}
    - \diver ( D \grad u) = S \quad \mathrm{in} \, \, \Omega,
\)
where $S = S(x)$ represents the prescribed, time-independent distribution of sources and sinks of the system. We equip \eqref{eq:conservation_law}
with the Dirichlet boundary condition
\(
    \label{eq:boundary_u}
    u = c \quad \mathrm{on} \, \,  \partial\Omega. \notag
\)
Here $c\in\R$ is a constant representing the equilibrium of the system. Indeed, we assume $\Phi$ to have a critical point at $c$, i.e.,
\(
    \label{eq:equilibrium_state}
    \Phi'(c) = 0. 
\)

The formal $L^2$-gradient flow of the energy functional \eqref{eq:energy_functional1} constrained by \eqref{eq:conservation_law} is of the form
\(
    \label{eq:GF}
    \part{D}{t} = \Phi''(u) \nabla u \otimes \nabla u + \frac{\nabla\sigma\otimes\nabla u + \nabla u\otimes\nabla\sigma}{2} \qquad \mbox{in } (0, \infty) \times \Omega,
\)
see Lemma 1 of \cite{portaro2023} for details.
Here $t\geq 0$ is the time-like variable induced by the gradient flow and $\sigma = \sigma (x, t)$ is the solution of the boundary value problem
\(
    \label{eq:sigma_law}
    - \diver (D \nabla \sigma) = \Phi'''(u) \nabla u \cdot D \nabla u \qquad \mathrm{in} \, \, (0, \infty) \times \Omega,
\)
subject to the homogeneous Dirichlet boundary condition $\sigma = 0$ on $\partial \Omega$.

In this paper we focus on the spatially one--dimensional setting of (\ref{eq:GF}--\ref{eq:sigma_law}). In particular, we shall formulate sufficient conditions for global existence of solutions. Moreover, we provide insights into global and local well-posedness propeties of the system by discussing several special cases and constructing relevant examples. The last part of the paper is dedicated to an appropriate discretization of 
(\ref{eq:GF}--\ref{eq:sigma_law}) presentation of several numerical examples.

While investigating a numerical scheme that solves the issue of different scales in time, a solution has been proposed that treats a specific component of the system implicitly, while keeping the rest explicit. By adopting this approach, it becomes possible to create a time discretization wherein the implicit portion of the system can be inverted relatively easily, often bypassing the need for solving nonlinear systems. At the same time, the explicit part maintains its accuracy and capability to capture the stiffness. The numerical scheme that addresses these issues is the IMEX Runge--Kutta scheme \cite{IMEX, pareschi2000implicit,pareschi2005implicit} and it becomes possible to develop high-order linearly implicit schemes using the conventional Runge--Kutta methods. They are generally associated to a Butcher's tableau, that is composed by a matrix of coefficients $A = \{a_{i,j} \}$, a vector of weights $b_i$, and a vector of nodes $c_i$.

In this paper we shall employ this methodology to obtain efficient semi--implicit discretizations that offer high order accuracy in the framework of biological transportation structures. IMEX schemes are proven to be very effective for many applications, e.g., for Navier-Stokes equations \cite{boscheri2021high}, Euler equations \cite{avgerinos2019linearly} and for generic high order PDEs \cite{sebastiano2023high}. In this paper we make use of a third--order semi--implicit scheme in time for the spatially one--dimensional version of (\ref{eq:GF}--\ref{eq:sigma_law}).

%%%%%%%%%%%%%%%%%%%%%%%%%%%%%%%%%%%%%%%%%%%%%%%%%%%%%%%%%%%%%%%%%%%%%%%%%%%%%%%%%%%%%%%%%%%%%%%%
\section{The spatially 1D setting}
Without loss of generality, let $\Omega = (0, 1)$ and the equilibrium $c=0$. The one--dimensional version of \eqref{eq:conservation_law}, \eqref{eq:sigma_law} and \eqref{eq:GF} reads 
\begin{subequations}
\begin{numcases}{}
    -(D u_x)_x = S,  \label{eq:conservation_law_1D} \\
    -(D \sigma_x)_x = \Phi'''(u) D u_x^2, \label{eq:sigma_law_1D} \\
    D_t = \Phi''(u) u_x^2 + u_x \sigma_x, \label{eq:GF_1D}
\end{numcases}
\end{subequations}
for all $(x, t) \in (0,1) \times (0, \infty)$,
where $u_x$ denotes the derivative of $u$ with respect to the spatial variable. The system is
equipped with homogeneous Dirichlet boundary conditions for $u$ and $\sigma$ and with uniformly positive and bounded initial condition
\(  
    \label{eq:initial_condition}
    D(0, x) = D_I (x) \ge \gamma > 0 \quad \forall x \in [0, 1], \quad D_I \in L^{\infty}(0, 1).
\)
To simplify the notation, we will denote the explicit dependence of the quantities of interest on either the spatial coordinate $x$ or temporal coordinate $t$ only when the quantity in question is dependent solely on one of these variables. Otherwise, we will omit it.

System (\ref{eq:conservation_law_1D}--\ref{eq:GF_1D}) can be formulated in an integro-differential form. By integrating \eqref{eq:conservation_law_1D} with respect to $x$ we get
\begin{align}
    \label{eq:ux_1}
    u_x = \frac{\alpha(t) - R(x)}{D},
\end{align}
where $\alpha=\alpha(t)$ is constant in space and
\(
    \label{eq:R}
    R(x) := \int_0^x S(y) \d y. \notag
\)
Another subsequent integration of (\ref{eq:ux_1}) in $\Omega$, exploiting the homogeneous Dirichlet boundary condition for $u$, leads to
\begin{align}
    \alpha(t) = \frac{\bigintsss_0^1 \frac{R(y)}{D} \d y}{\bigintsss_0^1 \frac{\d y}{D} }. \notag
\end{align}
Moreover, restricting the integration to $(0, x)$, we get the explicit expression for $u$,
\begin{align}
    \label{eq:u_1}
    u = \alpha(t) \int_0^x \frac{\d y}{D(y)} - \int_0^x \frac{R(y)}{D} \d y.
    %\notag \\ &= \int_0^1 \frac{R(y)}{D} \d y \frac{\bigintsss_0^x \frac{\d y}{D} }{\bigintsss_0^1 \frac{\d y}{D} } - \int_0^x \frac{R(y)}{D} \d y.
\end{align}
Performing a similar reasoning for \eqref{eq:sigma_law_1D}, we obtain
\begin{align}
    -D \sigma_x + \beta(t) &= \Phi''(u) D u_x - \int_0^x \Phi''(u) (D u_x)_x \d y \notag \\
    &= \Phi''(u) D u_x + \int_0^x \Phi''(u) S(y) \d y, \notag
\end{align}
which can be reformulated in the more convenient way
\(
    \label{eq:sigma_x_1}
    \sigma_x = \frac{\beta(t)}{D} - \Phi''(u) u_x - \frac{1}{D} \int_0^x \Phi''(u) S(y) \d y,
\)
with
\begin{align}
    \beta(t) = \frac{\bigintsss_0^1 \frac{1}{D} \bigintsss_0^y \Phi''(u) S(z) \d z \d y  }{ \bigintsss_0^1 \frac{\d y}{D} }, \notag
\end{align}
obtained through an integration of \eqref{eq:sigma_x_1} and employing \eqref{eq:equilibrium_state}.
Finally, substitution of \eqref{eq:ux_1} and \eqref{eq:sigma_x_1} into \eqref{eq:GF_1D} gives
\begin{align}
    D_t % &= \beta(t) \frac{u_x}{D} - \frac{u_x}{D} \int_0^x \Phi''(u) S(y) \d y \notag \\
    &= - \frac{u_x}{D} \left( \int_0^x \Phi''(u) S(y) \d y - \beta(t) \right) \notag \\
    &= \left( \frac{R(x) - \alpha(t)}{D^2} \right) \left( \int_0^x \Phi''(u) S(y) \d y - \beta(t) \right), \notag
\end{align}
which we recast as
\begin{align}
    \label{eq:Dt_1}
    D^2 D_t &= \left( \frac{ \bigintsss_0^1 \frac{R(x) - R(y)}{D} \d y } { \bigintsss_0^1 \frac{\d y}{D}} \right) \left( \frac{ \bigintsss_0^1 \frac{ \bigintsss_0^x \Phi''(u) S(z) \d z - \bigintsss_0^y \Phi''(u) S(z) \d z}{D} \d y } { \bigintsss_0^1 \frac{\d y}{D}} \right) \notag \\
    &= \left( \frac{ \bigintsss_0^1 \frac{\bigintsss_y^x S(z) \d z}{D} \d y } { \bigintsss_0^1 \frac{\d y}{D}} \right) \left( \frac{ \bigintsss_0^1 \frac{\bigintsss_y^x \Phi''(u) S(z) \d z}{D} \d y } { \bigintsss_0^1 \frac{\d y}{D}} \right).
\end{align}
In order to give a meaning to the elliptic equations (\ref{eq:conservation_law_1D}--\ref{eq:sigma_law_1D}) we look for a solution $D > 0$.

It is immediate to notice that $u$ is bounded in terms of the $L^1$ norm of $\frac{1}{D}$: 
\begin{proposition}\label{proposition_ubounded}
    Let $u$ be given by equation \eqref{eq:u_1}. Then for all $t\geq0$,
    \begin{align}
        \label{eq:ubound_1/d}
        \| u(\cdot, t) \|_{L^{\infty}(0,1)} \le \| S \|_{L^1 (0,1)} \bigg{\|} \frac{1}{D(\cdot, t)} \bigg{\|}_{L^1 (0,1)}.
    \end{align}
    \begin{proof}
        Since $u=0$ for $x\in \{ 0,1 \}$, it follows that the function $| u |$ attains its maximum at some point $x_m \in (0,1)$.
        Then, due to the differentiability of $u$, we have $u_x = 0$ at $x = x_m$, so that
        \begin{align}
            R(x_m) = \frac{\bigintsss_0^1 \frac{R(y)}{D}\d y}{\bigintsss_0^1 \frac{\d y}{D}}. \notag
        \end{align}
        Then, for some $\xi(t) \in (0, 1)$ we get the following estimate
        \begin{align}
            | u(x_m) | &= \bigg| R(x_m) \bigintsss_0^{x_m} \frac{\d y}{D} - \bigintsss_0^{x_m} \frac{R(y)}{D} \d y \bigg| \notag \\
            &= \bigg| \left( R(x_m) - R(\xi(t) \right) \bigintsss_0^{x_m} \frac{\d y}{D} \bigg| \notag \\
            &\le \| S \|_{L^1 (0, 1)} \bigintsss_0^1 \frac{1}{|D|} \d y, \notag
        \end{align}
        which concludes the proof.
    \end{proof}
\end{proposition}

Furthermore, if $\Phi''$ is bounded on $\R$, it is possible to obtain a uniform bound (on every finite time interval) for $\Phi'(u)$ which is independent of $\int_0^1 \frac{\d y}{D(y)}$.

\begin{proposition}\label{boundness_phi'}
    Let $S \in L^1(0, 1)$, $D_I \in L^1_+ (0, 1)$ with \eqref{eq:initial_condition} and $\Phi'' \in L^{\infty}(\R)$.
    If $D(x, t) \ge 0$ for $(x, t) \in [0, 1] \times [0, T]$, with $0<T<\infty$, then there exists $C(T) > 0$ such that
    \(
        \label{eq:boundness_phi'}
        \| \Phi'(u(\cdot, t)) \|_{L^{\infty}(0,1)} \le C(T) \quad \forall t \in [0, T].
    \)
    \begin{proof}
        Fix $\eps > 0$. Using \eqref{eq:ux_1} and the convexity of $\Phi$, we can recast the energy functional \eqref{eq:energy_functional1} as
        \begin{align}
            E[D(t)] &= \int_0^1 D u^2_x \Phi''(u) \d x \notag \\
            &= \int_0^1 |D u_x| | u_x \Phi''(u) | \d x \notag \\
            &= \int_0^1 |\alpha(t) - R(x) | \big| \Phi'(u)_x \big| \d x \notag \\
            &= \int_{|\alpha(t) - R(x)| \le \eps} |\alpha(t) - R(x) | \big|  \Phi'(u)_x \big| \d x + \int_{|\alpha(t) - R(x)| > \eps} |\alpha(t) - R(x) | \big| \Phi'(u)_x \big| \d x \notag \\
            &=: I_{\eps_-} + I_{\eps_+}. \notag
        \end{align}
    Since $D=D(t)$ represents a gradient flow of the energy $E$,  we have $E[D(t)] \le E[D_I] < \infty$ for all $t \ge 0$, and the equality holds if and only if $D \equiv D_I$. Therefore,
    \begin{align}
        I_{\eps_+} := \int_{|\alpha(t) - R(x)| > \eps} |\alpha(t) - R(x) | \big| \Phi'(u)_x \big| \d x < E[D_I], \notag
    \end{align}
    and, consequently,
    \begin{align}
        \label{eq:bound1_phi'}
        \int_{|\alpha(t) - R(x)|>\eps} \big| \Phi'(u)_x \big| \d x < \frac{E[D_I]}{\eps}.
    \end{align}
    Now, take $x \in \{ z \in [0, 1]; \, | \alpha(t) -R(z) | \le \eps \}$. Then, from \eqref{eq:Dt_1} we obtain the estimate
    \begin{align}
        \frac{1}{3} \big| (D^3)_t \big| &= \bigg| \left[ R(x) - \alpha(t) \right] \left( \int_0^x \Phi''\left( u \right) S(y) \d y - \beta(t) \right) \bigg| \notag \\
        &\le 2 \eps \| \Phi'' \|_{L^{\infty}(0,1)} \| S \|_{L^1(0,1)}. \notag
    \end{align}
    Integrating the latter expression in time between $(0, t)$ and choosing 
    \(
        \eps = \frac{7}{48} \frac{\gamma^3}{\| \Phi'' \|_{L^{\infty}(0,1)} \| S \|_{L^1(0,1)} T} \notag
    \)
    we get the following lower bound for $D(x,t)$
    \begin{align}
        D^3 (x, t) \ge D_I(x) - 6 \eps \| \Phi''(u) \|_{L^{\infty}(0,1)} \| S \|_{L^1(0,1)} t \notag \ge \frac{1}{8} \gamma^3  \notag
    \end{align}
    for all $(x,t) \in \{ z \in [0, 1]; \, | \alpha(t) - R(z) | \le \eps \} \times [0,T]$.
    Therefore, from the latter inequality and \eqref{eq:ux_1} we find the following bound for $u_x$,
    \begin{align}
        |u_x| \le \frac{2 \eps}{\gamma} \qquad \mbox{for all } (x,t) \in \{ z \in [0, 1]; \, | \alpha(t) - R(z) | \le \eps \} \times [0,T], \notag
    \end{align}
    and
    \begin{align}
        \label{eq:bound2_phi'}
        \int_{|\alpha(t) - R(x)| \le \eps} \big| \Phi'(u)_x \big| \d x \le \frac{2 \eps \| \Phi'' \|_{L^{\infty}(0,1)}}{\gamma} \qquad \mbox{for all } t \in [0, T].
    \end{align}
    A combination of \eqref{eq:bound1_phi'} and \eqref{eq:bound2_phi'} yields
    \begin{align}
        \int_0^1 \big| \Phi'(u)_x \big| \d x < \frac{E[D_I]}{\eps} + \frac{2 \eps \| \Phi'' \|_{L^{\infty}(0,1)}}{\gamma} \qquad \mbox{for all } t \in [0, T]. \notag
    \end{align}
    Finally, since $\Phi'\left( u(x,t) \right) = \Phi'(0) = 0$ for $x \in \{0, 1 \}$  we conclude that
    \begin{align}
        \| \Phi'\left( u(\cdot, t) \right) \|_{W^{1,1}(0,1)} \le \Tilde{C}(T) \qquad \mbox{for all } t\in [0, T], \notag
    \end{align}
    so that
    \begin{align}
        \| \Phi'\left( u(\cdot, t) \right) \|_{L^{\infty}(0,1)} \le C(T) \qquad \mbox{for all } t \in [0, T]. \notag
    \end{align}
    \end{proof}
\end{proposition}

In particular, Proposition \ref{boundness_phi'} implies that, if $\Phi'' \in L^{\infty}(\R)$, then $\|u(\cdot, t) \|_{L^{\infty}(0, 1)}$ is locally bounded in $t$ if $\Phi'(\infty) = \infty$ and $\Phi'(-\infty) = -\infty$.

The following theorem provides existence and uniqueness of a solution of the integro-differential system (\ref{eq:u_1}--\ref{eq:Dt_1}).

\begin{theorem} \label{Theorem:existence_1D}
    The system (\ref{eq:u_1}--\ref{eq:Dt_1}) complemented with the initial condition \eqref{eq:initial_condition} has a unique positive solution $D \in C^1 \left( [0, T); C [0, 1] \right)$. The maximal time of existence $T>0$ is finite if and only if $\liminf_{t \rightarrow T^-} min_{x \in [0, 1]} D(x,t) = 0$, otherwise the solution is global with $T=+\infty$.
\end{theorem}

The proof of the theorem is a direct consequence of the Cauchy-Lipschitz-Picard theorem. First, a bound for $u$ depends only on the minimum of $D$ from Proposition \ref{proposition_ubounded}. Employing Picard iterations (or Banach contraction) in the space of continuous functions for a sufficiently small time $T$, such that $D$ remains bounded from below by a positive constant, which implies existence of a local solution. The solution is global in time if $D=D(x,t)$ remains positive for all $t\geq 0$.

To facilitate computations and gain deeper insight into the system, we now replace the homogeneous Dirichlet boundary conditions with mixed Neumann-Dirichlet,
\begin{align}
    u_x (t, 0) = 0, &\quad u(t, 1) = 0, \notag \\
    \sigma_x(t, 0) = 0, &\quad \sigma(t, 1) = 0. \notag
\end{align}
This modification does not alter neither the gradient flow computations nor Proposition \ref{proposition_ubounded} and Theorem \ref{Theorem:existence_1D}. Moreover, the resulting integro-differential system is a straightforward modification of \eqref{eq:Dt_1}, \eqref{eq:u_1}:
\begin{subequations} \label{eq:system1D}
\begin{numcases}{}
    D^2 D_t = R(x) \int_0^x \Phi''(u) S(y) \d y,  \label{eq:Dt_2} \\
    u = \int_x^1 \frac{R(y)}{D} \d y \label{eq:u_2},
\end{numcases}
\end{subequations}
complemented with the initial condition \eqref{eq:initial_condition}, and the energy functional \eqref{eq:energy_functional1} can be written as
\(  
    \label{eq:energy_functional_1D}
    E[D] = \int_0^1 \frac{R(x)^2}{D} \Phi''(u) \d x.
\)

We now focus on identifying conditions under which the diffusivity $D(x,t)$ remains positive for all $t>0$,
implying global well-posedness of the problem by Theorem \ref{Theorem:existence_1D}. The following proposition provides some examples.

\begin{proposition} \label{proposition_D>0}
    Let $D$ be the unique solution of \eqref{eq:Dt_2}, \eqref{eq:u_2} with the initial condition \eqref{eq:initial_condition}. If any of the following conditions hold:
    \begin{enumerate}
        \item \label{sign_S} $S(x) \ge 0$ or $S(x) \le 0$ almost everywhere on $(0, 1)$,
        \item \label{sign_R} $R(x) \ge 0$ for $x \in [0, 1]$ and $\Phi'''(u) \ge 0$, or $ R(x) \le 0$ for $x \in [0, 1]$ and $\Phi'''(u) \le 0 $.
        \item %Write $S(x) = f(x) - g(x)$ for $x \in (0, 1)$ with $f,g \in L^1(0,1)$ non-negative functions.
        Define $S^+(x):=\max\{S(x),0\}$ and, resp., $S^-(x):=\max\{-S(x),0\}$ the positive and, resp., negative parts of $S=S(x)$.
        If there exists constants $K_+, K_- > 0$ such that $K_- \le \Phi''(u) \le K_+ \, \, \forall u \in \R$ and 
        \(
            \label{eq:condition1}
            \int_0^x S^+(y) \d y \ge \frac{K_+}{K_-} \int_0^x S^-(y) \d y, \quad \forall x \in (0,1)
        \)
        or
        \(
            \label{eq:condition2}
            \int_0^x S^+(y) \d y \le \frac{K_-}{K_+} \int_0^x S^-(y) \d y, \quad \forall x \in (0,1),
        \)
    \end{enumerate}
    then there exists a $\gamma > 0$ such that $D(x,t) \ge \gamma$ for all $(x,t) \in [0, 1] \times [0, \infty)$.

\begin{proof}
If \eqref{sign_S} holds, then trivially $D_t \ge 0$ and the result follows. \\
Assume now \eqref{sign_R}. Integrating \eqref{eq:Dt_2} by parts and using $u_x = - \frac{R(x)}{D}$, we obtain
\begin{align}
    D^2 D_t &= R(x) \int_0^x \Phi''(u) R'(y) \d y \notag \\
    &= R(x) \left( \Phi''(u) R(x) - \int_0^x \Phi'''(u) u_x R(y) \d y \right) \notag \\
    &= R^2(x) \Phi''(u) + R(x) \int_0^x \Phi'''(u) \frac{R(y)^2}{D} \d y \ge 0 \notag,
\end{align}
which yields the statement.
Finally, let \eqref{eq:condition1} hold. Then,
\begin{align}
    R(x) &= \int_0^x \big( S^+(y) - S^-(y) \big) \d y \notag \\
    &\ge \left( \frac{K_+}{K_-} - 1 \right) \int_0^x S^-(y) \d y \ge 0 \quad \forall x \in (0, 1), \notag
\end{align}
and, consequently
\begin{align}
    D^2 D_t \ge R(x) \int_0^x \big( K_- S^+(y) - K_+ S^-(y) \big) \d y \ge 0 \quad \forall x \in (0, 1), \notag
\end{align}
which proves the proposition. An analogous argument can be made in the case where \eqref{eq:condition2} holds.
\end{proof}
\end{proposition}

Let us remark that by a simple modification of the proof of Proposition \ref{boundness_phi'},
one obtains $\| u(\cdot, t) \|_{L^{\infty}(0, 1)}$ locally bounded in $t$ also for the case of mixed boundary conditions, whenever $\Phi'' \in L^{\infty}(\R)$ and $\Phi'(\infty) = \infty$, $\Phi'(-\infty) = -\infty$.

To examine the convexity properties of the functional \eqref{eq:energy_functional_1D} under the constraint \eqref{eq:system1D}, we need to calculate its second-order variation.

\begin{proposition} \label{proposition_secondvariation}
    The second-order variation of \eqref{eq:energy_functional_1D} coupled to \eqref{eq:system1D} reads
    \begin{align}
    \label{eq:second_variation1D}
        \frac{\delta^2 E[D^0]}{\delta D^2} (D^1, D^1) = \bigintssss_0^1 \frac{R^2(x)}{D^0} \bigg[ &2 \Phi''(u^0) \frac{(D^1)^2}{(D^0)^2} + 2 \Phi'''(u^0) \left( \frac{D^1}{D^0} \bigintssss_x^1 \frac{R(y) D^1}{(D^0)^2} \d y + \bigintssss_x^1 \frac{R(y) (D^1)^2}{(D^0)^3} \d y \right) \notag \\
        &+ \Phi^\textrm{\romannumeral 4}(u^0) \left( \bigintssss_x^1 \frac{R(y) D^1}{(D^0)^2} \d y \right)^2 \bigg] \d x. 
    \end{align}
    Moreover, if 
    \begin{enumerate}
        \item \label{convexity_hyp_1} $2 \Phi''(u) \Phi^\textrm{\romannumeral 4} (u) \ge \Phi'''(u)^2$ on $\R$,
        \item \label{convexity_hyp_2} $R(x) \ge 0$ on $(0, 1)$ and $\Phi'''(u) \ge 0$ on $\R$ or $R(x) \le 0$ on $(0, 1)$ and $\Phi'''(u) \le 0$ on $\R$,
    \end{enumerate}
    then $E \equiv E[D]$ is convex on $L^2_{\gamma} (0, 1) := \{ D \in L^2(0, 1) \, \mbox{ such that } \, D \ge \gamma \, \mathrm{a. e.} \, \mathrm{on} \, (0, 1)\}$ for any $\gamma > 0$.
    \begin{proof}
        Fix $\eps \in \R$. We expand $D = D^0 + \eps D^1 + O(\eps^2)$ with $D^0 \ge \gamma > 0$ $\mathrm{a.e.}$ on $(0, 1)$. Taylor expansion up to second order gives
        \begin{align}
            u[D^0 + \eps D^1] &= u[D^0] + \eps \frac{\delta u[D^0]}{\delta D}(D^1) + \frac{\eps^2}{2} \frac{\delta^2 u[D^0]}{\delta D}(D^1, D^1) + O(\eps^3) \\
            &=: u^0 + \eps u^1 + \frac{\eps^2}{2} u^2 + O(\eps^3).
        \end{align}
        Plugging the above expansion into \eqref{eq:conservation_law_1D}, we can collect the $O(1)$ terms
        \begin{align}
        \label{eq:conservation_law_1D_0}
            -\left( D^0 u^0_x \right)_x = S,
        \end{align}
        the first-order terms
        \begin{align}
            -\left( D^0 u^1_x + D^1 u^0_x \right)_x = 0 \notag
        \end{align}        
        and the second-order ones
        \begin{align}
            -\left( \frac{1}{2} D^0 u^2_x + D^1 u^1_x \right)_x = 0. \notag
        \end{align}
        Integration in $x$ of the latter one--dimensional conservation laws  for $u_0, u_1, u_2$ leads to the explicit expressions
        \begin{subequations} \label{eq:u0u1u2}
        \begin{align}
            u^0_x = - \frac{R(x)}{D^0}, &\quad u^0 = \bigintssss_x^1 \frac{R(y)}{D^0} \d y \label{eq:u0} \\
            u^1_x = \frac{R(x) D^1}{(D^0)^2}, &\quad u^1 = - \bigintssss_x^1 \frac{R(y) D^1}{(D^0)^2} \d y \label{eq:u1} \\
            u^2_x = -2 \frac{R(x) (D^1)^2}{(D^0)^3}, &\quad u^2 = \bigintssss_x^1 2 \frac{R(y) (D^1)^2}{(D^0)^3} \d y. \label{eq:u2}
        \end{align}
        \end{subequations}
        Multiplication of \eqref{eq:conservation_law_1D} by $\Phi'(u)$ and integration by parts allow us to recast the energy functional \eqref{eq:energy_functional_1D} as
        \begin{align*}
            E[D] = \int_0^1 S(x) \Phi'(u) \d x.
        \end{align*}
        Thus, we have for the second variation of the functional
        \begin{align*}
            \frac{\delta^2 E[D^0]}{\delta D^2} (D^1, D^1) &= \frac{d^2}{d \eps^2} \int_0^1 S(x) \Phi'\left( u^0 + \eps u^1 + \frac{\eps^2}{2} u^2 \right) \d x \bigg|_{\eps=0} \notag \\
            &= \frac{d^2}{d \eps^2} \int_0^1 S(x) \bigg[ \Phi'(u^0) + \eps \Phi''(u^0) u^1 + \frac{\eps^2}{2}\left( \Phi''(u^0) u^2 + \Phi'''(u^0) (u^1)^2 \right) \bigg] \d x \bigg|_{\eps=0} \\
            &= \int_0^1 S(x) \left( \Phi''(u^0) u^2 + \Phi'''(u^0) (u^1)^2 \right) \d x.
        \end{align*}
        Using \eqref{eq:conservation_law_1D_0}, integration by parts and substitution of \eqref{eq:u0u1u2} results in
        \begin{align*}
            \frac{\delta^2 E[D^0]}{\delta D^2} (D^1, D^1) = \bigintssss_0^1 \frac{R^2(x)}{D^0} \Bigg[ &2 \Phi''(u^0) \frac{(D^1)^2}{(D^0)^2} + 2 \Phi'''(u^0) \left( \frac{D^1}{D^0} \bigintssss_x^1 \frac{R(y) D^1}{(D^0)^2} \d y + \bigintssss_x^1 \frac{R(y) (D^1)^2}{(D^0)^3} \d y \right) \notag \\
            &+ \Phi^\textrm{\romannumeral 4}(u^0) \left( \bigintssss_x^1 \frac{R(y) D^1}{(D^0)^2} \d y \right)^2 \Bigg] \d x, 
        \end{align*}
        which proves \eqref{eq:second_variation1D}.

        Observe that if \eqref{convexity_hyp_2} holds, we have
        \begin{align*}
            \frac{\delta^2 E[D^0]}{\delta D^2} (D^1, D^1) \ge \bigintssss_0^1 \frac{R^2(x)}{D^0} \Bigg[ &2 \Phi''(u^0) \frac{(D^1)^2}{(D^0)^2} + 2 \Phi'''(u^0) \frac{D^1}{D^0} \bigintssss_x^1 \frac{R(y) D^1}{(D^0)^2} \d y \notag \\
            &+ \Phi^\textrm{\romannumeral 4}(u^0) \left( \bigintssss_x^1 \frac{R(y) D^1}{(D^0)^2} \d y \right)^2 \Bigg] \d x.
        \end{align*}
        Defining
        \begin{align*}
            v = \begin{bmatrix}
                \frac{D^1}{D^0} & \bigintssss_x^1 \frac{R(y) D^1}{(D^0)^2} \d y 
                \end{bmatrix}, \quad
                Q = \begin{bmatrix}
                    2 \Phi''(u^0) & \Phi'''(u^0) \\
                    \Phi'''(u^0) & \Phi^\textrm{\romannumeral 4}(u^0)
                \end{bmatrix},
        \end{align*}
        we write
        \begin{align*}
            \frac{\delta^2 E[D^0]}{\delta D^2} (D^1, D^1) \ge \bigintssss_0^1 \frac{R^2(x)}{D^0} v \cdot Q v \d x.
        \end{align*}
        Finally, if \eqref{convexity_hyp_1} holds then $\det Q \ge 0$ and $E[D]$ is convex on $L^2_{\gamma}(0, 1)$ for every $\gamma > 0$.
    \end{proof}
\end{proposition}

When considering the case of $\Phi(u) = u^p$, the energy functional is convex if $1 \le p \le 2$ and $R \le 0$, or if $p \ge 4$ and $R \ge 0$. Condition \eqref{convexity_hyp_1} is reminiscent, but not identical, of the one in \cite{AMT}. We emphasize that Condition \ref{convexity_hyp_2} satisfies the hypothesis of Proposition \ref{proposition_D>0}, ensuring that $D$ never reaches the boundary of $L^2_{\gamma}(0, 1)$ since $D_t \geq 0$.

\section{A remarkable example}

Based on Proposition \ref{proposition_D>0}, we have identified certain cases where the diffusion $D$ remains positive over the entire domain so that the problem is well-posed thanks to Theorem \ref{Theorem:existence_1D}. Nevertheless, it is important to note that these cases represent only a partial enumeration of the possible scenarios, and there may exist other cases where D remains positive throughout the domain or it becomes zero somewhere. By systematically exploring the range of possible scenarios, we can gain a deeper understanding of the behavior of $D$ and its implications for the system under study. This involves analyzing the system subject to different distributions of sources and sinks of $S(x)$. With this in mind, let us define
\(
    \label{eq:S_delta}
    S(x) := a \delta(x-x_0) - b \delta (x-x_1), \quad a,b > 0, \, \, 0 < x_0 < x_1 < 1,
\)
where $\delta(x - \bar{x})$ is the Dirac delta distribution centered at $\bar{x}$. Then
\(
    \label{eq:R_delta}
    R(x) = a \chi_{[x_0, x_1]}(x) + (a - b)  \chi_{[x_1, 1]}(x) \quad \forall x \in [0, 1].
\)
Substituting \eqref{eq:S_delta}, \eqref{eq:R_delta} into \eqref{eq:Dt_2}, we can recast the evolution equation for $D$ as
\begin{align}
    D^2 D_t &= \left[ a \chi_{[x_0, x_1]}(x) + (a - b)  \chi_{[x_1, 1]}(x) \right] \int_0^x \Phi''(u) \big( a \delta(y-x_0) - b \delta (y-x_1) \big) \d y, \notag
\end{align}
i.e.,
\begin{eqnarray} \label{eq:Dt_detailed_delta_1D}
    D^2 D_t = \left\{
\begin{array}{ll} \displaystyle
 0 &  x \in [0, x_0), \\
 a^2 \, \Phi'' ( u(x_0,t)) &  x \in [x_0, x_1), \label{eq:Dt_x0} \\
 (a - b) \big[ a \Phi'' (u(x_0,t)) - b \Phi''(u(x_1, t)) \big] & x \in [x_1, 1]. \label{eq:Dt_x1}
\end{array}
\right.
\end{eqnarray}

% \begin{subequations} \label{eq:Dt_detailed_delta_1D}
% \begin{numcases}
%     {D^2 D_t =} 0 &\quad x \in [0, x_0), \\
%     a^2 \, \Phi'' ( u(x_0,t)) &\quad x \in [x_0, x_1), \label{eq:Dt_x0} \\
%     (a - b) \big[ a \Phi'' (u(x_0,t)) - b \Phi''(u(x_1, t)) \big] &\quad x \in [x_1, 1]. \label{eq:Dt_x1}
% \end{numcases}
%\end{subequations}
For the sake of simplicity of presentation we choose a piecewise constant initial datum:
\begin{eqnarray} 
    {D(0, x) =} \left\{
\begin{array}{ll} \displaystyle
 D_{I,0} & x \in [0, x_0), \\
 D_{I, 1} & x \in [x_0, x_1), \\
 D_{I, 2} & x \in [x_1, 1],
\end{array}
\right.
\end{eqnarray}

% \begin{numcases}
%     {D(0, x) =} D_{I,0} &\quad x \in [0, x_0), \notag \\
%     D_{I, 1} &\quad x \in [x_0, x_1), \notag \\
%     D_{I, 2} &\quad x \in [x_1, 1], \notag
% \end{numcases}
with $D_{I,0},D_{I,1},D_{I,2}>0$.
We immediately notice that in the subintervals $[0, x_0)$, $[x_0, x_1)$ and $[x_1, 1]$ the evolution of $D$ does not depend on $x$. Hence
\begin{eqnarray} 
    {D(t,x) =} \left\{
\begin{array}{ll} \displaystyle
 D_{I,0} & x \in [0, x_0), \\
 D_{1}(t) & x \in [x_0, x_1), \\
 D_{2}(t) & x \in [x_1, 1],
\end{array}
\right.
\end{eqnarray}
% \begin{numcases}
%     {D(x,t) = }  D_{I, 0} &\quad x \in [0, x_0) \notag \\
%     D_1 (t) &\quad x \in [x_0, x_1) \notag \\
%     D_2 (t) &\quad x \in [x_0, x_1), \notag
% \end{numcases}
where $D_1(t) \ge \gamma > 0$  $\forall t \in [0, \infty)$ from \eqref{eq:Dt_x0}. Consequently, we rewrite \eqref{eq:u_2} as
\begin{eqnarray} 
    {u =} \left\{
\begin{array}{ll} \displaystyle
\frac{a(x_1 - x_0)}{D_1(t)} + \frac{(a -b) (1-x_1)}{D_2(t)}, & x \in [0, x_0), \\ \displaystyle
\frac{a(x_1 - x)}{D_1(t)} + \frac{(a-b) (1-x_1)}{D_2(t)} & x \in [x_0, x_1),
    \\ \displaystyle
\frac{(a-b) (1-x)}{D_2(t)} & x \in [x_1, 1],    
\end{array}
\right.
\end{eqnarray}
% \begin{numcases}
%     {u = } \frac{a(x_1 - x_0)}{D_1(t)} + \frac{(a -b) (1-x_1)}{D_2(t)}, &\quad x \in [0, x_0), \notag \\
%     \frac{a(x_1 - x)}{D_1(t)} + \frac{(a-b) (1-x_1)}{D_2(t)} &\quad x \in [x_0, x_1),
%     \notag \\
%     \frac{(a-b) (1-x)}{D_2(t)} &\quad x \in [x_1, 1], \notag
% \end{numcases}
where we are interested in 
\begin{subequations} \label{eq:u_delta}
\begin{align}
    \label{eq:u_x0_delta} u(x_0, t) &= \frac{a(x_1 - x_0)}{D_1(t)} + \frac{(a-b)(1-x_1)}{D_2(t)}, \\
    \label{eq:u_x1_delta} u(x_1, t) &= \frac{(a-b) (1-x_1)}{D_2(t)}, 
\end{align}
\end{subequations}
and we notice that $u(x_0, t) > u(x_1, t)$.

As we previously noted, existence and uniqueness of solutions is guaranteed whenever the function $D_2$ remains positive. Therefore, it is crucial to investigate conditions that lead to this result. To achieve that, we will leverage the properties of the gradient flow structure of our system.

\begin{proposition} \label{proposition_energy_delta}
    Let $D$ be the solution of \eqref{eq:Dt_detailed_delta_1D} with associated uniformly positive initial condition.
    
    In the case $a > b$ we have the following results.
    \begin{enumerate}
        \item \label{condition1.1} If $\int_0^{\infty} \Phi''(u) \d u \equiv \lim_{u \rightarrow \infty} \Phi'(u) = \infty$, then $D(x,t)>0$ $\forall (x, t) \in [0, 1] \times [0, \infty)$.
        \item \label{condition1.2} If $\int_0^{\infty} \Phi''(u) \d u \equiv \lim_{u \rightarrow \infty} \Phi'(u) =: \alpha < \infty$ and
        \(  
            \label{eq:condition_alpha}
            \alpha \ge \frac{E[D_I]}{(a-b)},
        \)
        then $D(x, t)>0$ $\forall (x, t) \in [0, 1] \times [0, \infty)$.
    \end{enumerate}
    We have the following results in the complementary case $a < b$.
    \begin{enumerate}
    \setcounter{enumi}{2}
        \item \label{condition2.1} If $\int_{-\infty}^0 \Phi''(u) \d u \equiv - \lim_{u \rightarrow -\infty} \Phi'(u) = \infty$, then $D(x, t)>0$ $\forall (x, t) \in [0, 1] \times [0, \infty)$.
        \item \label{condition2.2} If $\int_{-\infty}^0 \Phi''(u) \d u \equiv - \lim_{u \rightarrow -\infty} \Phi'(u) =: \beta < \infty$ and
        \(  
            \label{eq:condition_beta}
            \beta \ge \frac{E[D_I]}{(b-a)},
        \)
        then $D(x,t)>0$ $\forall (x, t) \in [0, 1] \times [0, \infty)$.
    \end{enumerate}
    
\end{proposition}

\begin{proof}
    From the gradient flow theory,  it follows that $E[D(t)] \le E[D_I] < \infty, \, \, \forall t \ge 0$ and the equality holds if and only if $D = D_I \, \, \forall t \ge 0$. Therefore, without loss of generality, we assume $E[D(t)] < E[D_I]$. Further, we can expand the energy functional \eqref{eq:energy_functional_1D} as
    \begin{align}
        E[D(t)] &= \int_{x_0}^{x_1} \frac{a^2}{D_1(t)} \Phi''(u) \d x + \int_{x_1}^1 \frac{(a-b)^2}{D_2(t)} \Phi''(u) \d x \notag \\
        &= \int_{x_0}^{x_1} \frac{a^2}{D_1(t)} \Phi''\left(\frac{a (x_1 - x)}{D_1(t)} + \frac{(a-b) (1 - x_1)}{D_2(t)}\right) \d x + \int_{x_1}^1 \frac{(a-b)^2}{D_2(t)} \Phi''\left(\frac{(a-b) (1 - x)}{D_2(t)}\right) \d x \notag \\
        &=: I_1(t) + I_2(t). \notag
    \end{align}
    Through a change of variable, it is possible to recast $I_1(t), I_2(t)$ as
    \begin{align} \displaystyle
        I_1(t) = a \int_{\frac{(a-b)(1-x_1)}{D_2(t)}}^{\frac{a (x_1 - x_0)}{D_1(t)} + \frac{(a-b)(1-x_1)}{D_2(t)}} \Phi''(y) \d y, \notag
    \end{align}
    and
    \begin{align} \displaystyle
        I_2(t) = (a-b) \int_0^{\frac{(a-b) (1-x_1)}{D_2(t)}} \Phi''(y) \d y. \notag
    \end{align}
    Assume the proposition is false, i.e., $\exists \, 0 < T < \infty$ such that $D_2(T) = 0$. \\
    Let $a > b$. \eqref{condition1.1} is easily proved since
    \(
        E[D(t)] = \infty < E[D_I], \notag
    \)
which is a contradiction. To prove \eqref{condition1.2}, we first notice that from hypothesis $\Phi'' \in L^1 (0, \infty)$, so that $I_1 (T) = 0$. Then
    \begin{align}
        (a-b) \alpha = E[D(t)] < E[D_I] \notag
    \end{align}
    leads to a contradiction of \eqref{eq:condition_alpha}. \\
    By applying analogous reasoning in the case where $a<b$, we arrive at a series of contradictions that establish the validity of both \eqref{condition2.1} and \eqref{condition2.2}.
\end{proof}

We observe that in the case $a > b$, we have
\(
    E[D(t)] = a \int_0^{\frac{a(x_1 - x_0)}{D_1(t)} + \frac{(a-b)(1- x_1)}{D_2(t)}} \Phi''(y) \d y - b \int_0^{ \frac{(a-b)(1- x_1)}{D_2(t)}} \Phi''(y) \d y < a \alpha \quad \forall t \ge 0. \notag
\)
Hence, proposition \ref{proposition_energy_delta} does not cover the cases
\(
\label{eq_condition_energy_ini}
    (a - b) \alpha < E[D_I] < a \alpha,
\)
for which we currently do not have results. On the other hand, whenever $a < b$, we write
\(
    E[D(t)] = (b - a) \int_{ - \frac{(b-a)(1-x_1)}{D_2(t)}}^{0} \Phi''(y) \d y + a \int_{- \frac{(b-a)(1-x_1)}{D_2(t)}}^{\frac{a (x_1 - x_0)}{D_1(t)} - \frac{(b-a)(1-x_1)}{D_2(t)}} < b \beta + a \alpha, \notag
\)
which leaves open the case
\(
    (b - a) \beta < E[D_I] < b \beta + a \alpha. \notag
\)

\begin{remark}
    The work \cite{portaro2023} highlights an important convex entropy generating function $\Phi(u) = (u+1) \left( \ln (u+1) - 1 \right)$ for $u > -1$. This function transforms the energy functional \eqref{eq:energy_functional1} into the Fisher information. When considering the one--dimensional setting and selecting $S(x)$ as in \eqref{eq:S_delta}, it is necessary to ensure that $a > b$ to maintain $u > -1$ for arbitrary values of $D_1$, $D_2$. Consequently, Condition \eqref{condition1.1} can be used to establish the global well-posedness of the Fisher problem.
\end{remark}

Let us now consider the evolutionary equation for $D_2(t)$, i.e., \eqref{eq:Dt_x1}. Dividing it by $\Phi''(u(x_1, t))$, we obtain
\begin{align}
    \frac{D^2_2 (D_2)_t}{\Phi'' \left( \frac{(a-b)(1 - x_1)}{D_2(t)}\right)} = (a - b) \left[ a \frac{\Phi'' \left( \frac{a (x_1 - x_0)}{D_1(t)} + \frac{(a-b)(1 - x_1)}{D_2(t)}\right)}{\Phi'' \left( \frac{(a-b)(1 - x_1)}{D_2(t)}\right)} - b \right] \notag
\end{align}
and integration between $0$ and $t$ yields
\begin{align}
    \bigintsss_{D_{I, 2}}^{D_2(t)} \frac{y^2 \d y}{\Phi'' \left( \frac{(a-b)(1 - x_1)}{y}\right)} &= (a - b) \left[ a \bigintss_0^t \frac{\Phi'' \left( \frac{a (x_1 - x_0)}{D_1(s)} + \frac{(a-b)(1 - x_1)}{D_2(s)}\right)}{\Phi'' \left( \frac{(a-b)(1 - x_1)}{D_2(s)}\right)} \d s - b t  \right] \notag \\
    &=: (a - b) (a I(t) - b t), \label{eq:integral_delta_1D}
\end{align}
where 
\begin{align}
    \label{eq:I_delta_1D}
    I(t) := \bigintss_0^t \frac{\Phi'' \left( \frac{a (x_1 - x_0)}{D_1(s)} + \frac{(a-b)(1 - x_1)}{D_2(s)}\right)}{\Phi'' \left( \frac{(a-b)(1 - x_1)}{D_2(s)}\right)} \d s.
\end{align}
Here we have replaced $u(x_0, t)$, $u(x_1, t)$ by \eqref{eq:u_delta}.

We have the following collection of results.

\begin{proposition}\label{proposition_integrale^4_delta}
    Let $D$ be the solution of \eqref{eq:Dt_detailed_delta_1D} with associated uniformly positive initial condition and $\Phi$ be strictly convex.
    \begin{enumerate}
        \item \label{condition_3.1} If $a > b$ and 
        \begin{align}
        \label{eq:1integral^4_1D}
            \int_1^{\infty} \frac{\d v}{ \Phi''(v) v^4} = \infty,
        \end{align}
        then $D(x,t)>0$ $\forall (x, t) \in [0, 1] \times [0, \infty)$.
        \item \label{condition_3.2} If $a < b$, $\limsup_{w \rightarrow -\infty} \frac{\Phi''(z + w)}{\Phi''(w)}$ is a locally bounded function of $ z \in [0, \infty)$ and
        \begin{align}
            \int_1^{\infty} \frac{\d v}{\Phi''(- v) v^4} = \infty,
        \end{align}
        then $D(x,t)>0$ $\forall (x, t) \in [0, 1] \times [0, \infty)$.
        \item \label{condition_3.3} If $a > b$ and
        \begin{align}
            \label{eq:liminf_delta_1D}
            \liminf_{w \rightarrow \infty} \frac{\Phi''(z+w)}{\Phi''(w)} > \frac{b}{a}
        \end{align}
        uniformly for $z$ in compact subsets of $(0, \infty)$, then $D(x,t)>0$ $\forall (x, t) \in [0, 1] \times [0, \infty)$.
        \item \label{condition_3.4} If $a < b$ and
        \begin{align}
            \label{eq:limsup_delta_1D}
            \limsup_{w \rightarrow -\infty} \frac{\Phi''(z+w)}{\Phi''(w)} < \frac{b}{a}
        \end{align}
        uniformly for $z$ in compact subsets of $(0, \infty)$, then $D(x,t)>0$ $\forall (x, t) \in [0, 1] \times [0, \infty)$.
    \end{enumerate}
    \begin{proof}
        Assume $D_2(t) > 0 $ on $[0, T)$ and $D_2(T) = 0$. Consequently, we rewrite \eqref{eq:integral_delta_1D} as
        \begin{align}
            \label{eq:contradiction_integral_1D}
            - \bigintsss_{0}^{D_{I, 2}} \frac{y^2 \d y}{\Phi'' \left( \frac{(a-b)(1 - x_1)}{y}\right)} = (a - b) \left( a I(T) - b T \right),
        \end{align}
        where $I(T) > 0$ from the hypothesis of strictly convexity of $\Phi$. On the other hand, we notice that
        \begin{align}
            - \bigintsss_{0}^{D_{I, 2}} \frac{y^2 \d y}{\Phi'' \left( \frac{(a-b)(1 - x_1)}{y}\right)} = - \bigintss_{\frac{(a-b)(1-x_1)}{D_{I,2}}}^{\infty} \frac{(a-b)^3 (1 - x_1)^3}{\Phi''(v) v^4} \d v = - \infty \notag
        \end{align}
        from \eqref{eq:1integral^4_1D}. Therefore, we should have $I(T) = -\infty$ which contradicts the positivity of this term. Hence, $D_2(t) > 0 \, \, \forall t > 0$ and \eqref{condition_3.1} is proved. \\
        To prove \eqref{condition_3.2}, we use a similar argument. Indeed, \eqref{eq:contradiction_integral_1D} holds and
        \begin{align}
            - \bigintsss_{0}^{D_{I, 2}} \frac{y^2 \d y}{\Phi'' \left( - \frac{(b-a)(1 - x_1)}{y}\right)} = - \bigintss_{\frac{(b-a)(1-x_1)}{D_{I,2}}}^{\infty} \frac{(b-a)^3 (1 - x_1)^3}{\Phi''(- v) v^4} \d v = - \infty, \notag
        \end{align}
        which implies that $I(T) = \infty$.
        Next, let us define 
        \begin{align}
            w(t) &:= - \frac{(b - a)(1 - x_1)}{D_2(t)}, \notag \\
            z(t) &:= \frac{a (x_1 - x_0)}{D_1(t)}. \notag
        \end{align}
        Clearly, $z(t) \le z(0) < \infty$, and from the definition \eqref{eq:I_delta_1D} of $I(t)$ combined with the bounded limsup hypothesis we must have $I(T) < \infty$, which leads to a contradiction. \\
        From \eqref{eq:liminf_delta_1D}, \eqref{eq:I_delta_1D} we immediately notice that the right-hand side of \eqref{eq:contradiction_integral_1D} must be positive, and, conversely, the left-hand side is negative from the hypothesis of strict convexity of $\Phi(u)$ and this gives the contradiction to prove \eqref{condition_3.3}.
        
        The statement \eqref{condition_3.4} is proved using an analogous argument.
    \end{proof}
\end{proposition}

%As a consequence of proposition \ref{proposition_integrale^4_delta} we have the following immediate corollary.
%\begin{corollary}
%    Let $D(x, t)$ be the solution of \eqref{eq:Dt_detailed_delta_1D} with associated uniformly positive initial condition and $\Phi$ be strictly convex.
%    \begin{enumerate}
%        \item If $a > b$ and $\lim_{u \rightarrow \infty} \Phi''(u) = \alpha$ with $\alpha \in (0, \infty)$, then $D(x,t)>0$ $\forall (x, t) \in [0, 1] \times [0, \infty)$.
%        \item If $ a < b $ and $ \lim_{u \rightarrow -\infty} \Phi''(u) = \beta $ with $\beta \in (0, \infty)$, then $D(x,t)>0$ $\forall (x, t) \in [0, 1] \times [0, \infty)$.
%    \end{enumerate}
%\end{corollary}

\section{Numerical results}
In this section we present several numerical simulations of the spatially one--dimensional systems (\ref{eq:Dt_2}--\ref{eq:u_2}) and of the ODE presented in \eqref{eq:Dt_x1},
 where the solution only exists locally in time due to $D=D(x,t)$ touching zero in finite time.
We make use of a $3^{\rm rd}$ order in time numerical scheme, and we consider several entropy functions $\Phi$ to observe in which cases $D=D(x,t)$ vanishes in finite time.

A semi--implicit  discretization is adopted to achieve high order of accuracy in time. In particular, we make use of implicit-explicit (IMEX) Runge--Kutta schemes \cite{pareschi2000implicit, IMEX}, which are multi-step methods based on s-stages and typically represented with the double Butcher tableau,
\begin{table}[H]
\centering
\begin{tabular}{c|c}
$\widetilde c$ & $\widetilde A$ \\ \hline 
 & $\widetilde b^\top$
\end{tabular} \qquad
\begin{tabular}{c|c}
    $c$ & $\widehat A$ \\ \hline
     & $\widehat b^\top$
\end{tabular}
\end{table}
with the matrices $(\widetilde A, \widehat A) \in \mathbb R^{s\times s}$ and the vectors $(\widetilde b, \widehat b, \widetilde c, \widehat c) \in \mathbb R^s$ where $s>0$ is the number of stages of the Runge--Kutta method. The tilde symbol refers to the explicit scheme, and $\widetilde A = \{ \widetilde a_{i,j} \}$ is a
lower triangular matrix with zero elements on the diagonal, while $\widehat A = \{ \widehat a_{i,j} \}$ is a
triangular matrix which accounts for the implicit scheme, thus having non-zero elements on the diagonal. Here, we adopt IMEX schemes with $\widetilde b = \widehat b$, and the stiffly accurate property in the implicit part.

%%%%%%%%%%%%%%%%%%%%%%%%
%%%%%%%%%%%%%%%%%%%%%%%

First, we discretize in space and time the system in (\ref{eq:Dt_2}--\ref{eq:u_2}), and after we focus on the case \eqref{eq:S_delta} in which the sink/source distribution $S=S(x)$ is given by a superposition of Dirac delta distributions.

\subsection{Discretization in space and time}
In this section we discretize in space and time the spatially one--dimensional system (\ref{eq:Dt_2}--\ref{eq:u_2}), which we recall here for the sake of the reader,
\begin{align}
\label{eq_D}
    D^2D_t &= R(x)V(x), \quad x\in\Omega, \, t\in[0,t_{\rm fin}] \\
    \label{eq_u}
    u &= \int_x^1 \frac{R(y)}{D}\, dy
\end{align}
where $V(x) = \int_0^x\Phi''(u)S(y)\,dy$ and $R(x) = \int_0^xS(y)\,dy$. The expression that we choose for the sink/source function, $S(x)$, is the following
\begin{align}
\label{eq_S_deltas}
    S(x) = m x + q, \quad m, q \in \mathbb R,
\end{align}
where the parameters $m$ and $q$ are chosen in such a way the function $S=S(x)$ changes sign in $[0,1]$, while its primitive $R=R(x)$ remains non-negative for every $x$. To close the system, we prescribe the initial condition $D(t=0) = D_I \in \mathbb R$. 

For a fixed $N \in \mathbb{N}$ we 
discretize the spatial domain $\Omega = [0,1]$ using the equidistant grid \textbf{x} = \{ $\textbf{x}_i = ih;\, i = 0,\cdots,N; \, h = 1/N$\}.
The discrete quantities shall be represented by bold letters. In particular, ${\bf S} = \{{\bf S}_i \approx S({\bf x}_i), i = 1,\cdots, N\}$ represents the vector that approximates the function $S(x)$ in the points ${\bf x}_i$. Moreover, we define
\begin{align}
   % \textbf{S} = \beta \textbf{x} + \gamma \\
\nonumber    \textbf{R} = m \frac{\textbf{x}^2}{2} + q \textbf{x} \\ \label{eq_ui}
    0 = \textbf{u}_N = \textbf{u}({\bf x}_N) \approx u(x = 1) \\
    \textbf{u}_{N-i} = \sum_{j = 1}^i \textbf{I}^u_{N-j}\, \quad i = 1,\cdots,N \\
     0 = {\bf V}_0 = V({\bf x}_0) \approx V(x = 0) \\
    \textbf{V}_{i} = \sum_{j = 1}^i \textbf{I}^V_{j}\, \quad i = 1,\cdots,N
\end{align}
where $\textbf{I}^u_i = (\textbf{R}_i/\textbf{D}_i + \textbf{R}_{i+1}/\textbf{D}_{i+1})h/2$ and $\textbf{I}^V_i = ({\Phi''}_i\textbf{S}_i + {\Phi''}_{i-1}\textbf{S}_{i-1})h/2$.

Now we multiply and divide \eqref{eq_D} by $\textbf{D}$
\begin{equation}
\label{eq_2_discr}
    \textbf{D}_t = M(\textbf{D})\textbf{D}
\end{equation}
where 
\[M(\textbf{D}) = \frac{\textbf{R}(x)\textbf{V}(x)}{\textbf{D}^3}.\]

To apply the partitioned Runge--Kutta method to \eqref{eq_2_discr}, let us first set $\textbf{D}^1_E = \textbf{D}^n$, then the stage fluxes for $i = 1, \cdots, s$ are calculated as
\begin{align}
\label{eq_imex_DE}
    \textbf{D}_E^{i}& = \textbf{D}^n + \Delta t\,\sum_{j=1}^{\rm s-1}\widetilde a_{i,j}M(\textbf{D}_E^j)\textbf{D}_I^j, \quad i = 1,\cdots,s \\
\label{eq_imex_DI}
    \textbf{D}_I^{i} &= \textbf{D}^n + \Delta t\,\sum_{j=1}^{\rm s} \widehat a_{i,j}M(\textbf{D}_E^j)\textbf{D}_I^j, \quad i = 1,\cdots,s 
\end{align}
and the numerical solution is finally updated with
\begin{align}
\label{eq_imex_bi}
    \textbf{D}^{n+1} &= \textbf{D}^n + \Delta t\sum_{i=1}^{\rm s} \widehat b(i) M(\textbf{D}_E^i)\textbf{D}_I^i
\end{align}
where $\rm s$ is the number of stages of the scheme, and $\Delta t>0$ the time step.

%We are interested in different expressions for the entropy $\Phi$, considering
% \begin{itemize}
%     \item $\Phi''(u) = \exp(-u)$
%     \item $\Phi''(u) = (u+1)^{-1}$
%     \item $\Phi''(u) = (u^2(\sin(u)+\sigma)+1)^{-1}$
% \end{itemize}
% where $\sigma$ is chosen {such that} {convexity of $\Phi(u)$ is maintained}.

\subsubsection{Accuracy test and Results}
In this section we first check the accuracy of the discretization. We start with a simple case, with a specific choice of the second derivative of the entropy function, $\Phi''(u) = 1$, which enables us to find an exact solution (see Fig.~\ref{fig_accuracy_Du} (a)). Moreover, we test the accuracy of the time discretization in a general setting, with a reference solution calculated with a fine grid (see Fig.~\ref{fig_accuracy_Du} (b)). 

%The IMEX scheme used in this work is a Strong Stability Preserving (SSP) method. 
A generic IMEX Runge--Kutta scheme is described with a triplet $(\widehat s, \widetilde s, p)$ which characterizes the number $\widehat s$ of stages of the implicit scheme, the number $\widetilde s$ of stages of the explicit scheme and the order $p$ of the resulting scheme.

%A third order semi--implicit Runge--Kutta methods can be obtained  
A Runge--Kutta scheme is of order 3 if and only if the following conditions are satisfied \cite[Theorem 2.1]{hairer}:
\begin{align}
\label{eq_condition_3orderA}
    \sum_j b_j &= 1, \quad &  2\sum_{j,k} b_ja_{j,k} &= 1, \\
    \label{eq_condition_3orderB}
    3\sum_{j,k,l}b_j a_{j,k}a_{j,l} &= 1, \quad  & 6\sum_{j,k,l}b_ja_{j,k}a_{k,l} &= 1.
\end{align}
As third semi--implicit Runge–Kutta methods that satisfies the set of order conditions (\ref{eq_condition_3orderA}--\ref{eq_condition_3orderB})
is given by the IMEX-SSP3(4,3,3) L-stable Scheme: SSP-LDIRK3(4,3,3) with $ \widehat b = \widetilde b$, i.e.
\begin{table}[H]
\centering
\begin{tabular}{c|c c c c}
0 & 0 & 0 & 0 & 0 \\  
0 & 0 & 0 & 0 & 0 \\ 
1 & 0 & 1 & 0 & 0 \\  
1/2 & 0 & 1/4 & 1/4 & 0 \\ \hline 
 & 0 & 1/6 & 1/6 & 2/3 \\ 
\end{tabular} \qquad
\begin{tabular}{c|c c c c}
$\lambda$ & $\lambda$ & 0 & 0 & 0 \\  
0 & -$\lambda$ & $\lambda$ & 0 & 0 \\ 
1 & 0 & $1-\lambda$ & $\lambda$ & 0 \\  
1/2 & $\mu$ & $\eta$ & $1/2 - \mu - \eta - \lambda$ & $\lambda$ \\ \hline 
 & 0 & 1/6 & 1/6 & 2/3 \\ 
\end{tabular}
\end{table}
with $\lambda = 0.24169426078821$, $\mu = \lambda/4$ and $\eta = 0.12915286960590$ \cite{IMEX}.

\begin{figure}[ht]
\centering
\begin{minipage}[b]
		{.45\textwidth}
		\centering
	\begin{overpic}[abs,width=\textwidth,unit=1mm,scale=.25]{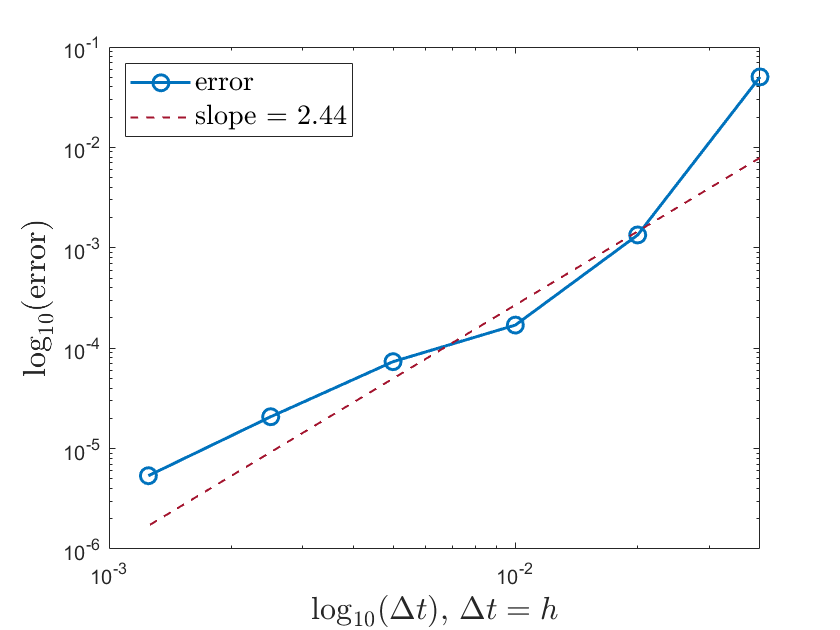}
 \put(1,54){(a)}
\end{overpic}
\end{minipage}
\begin{minipage}[b]
		{.45\textwidth}
		\centering
\begin{overpic}[abs,width=\textwidth,unit=1mm,scale=.25]{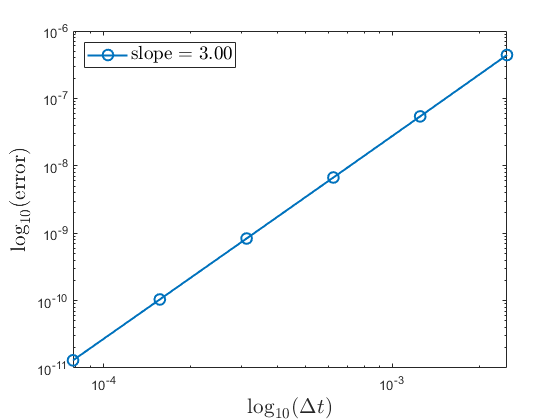}
\put(1,54){(b)}
\end{overpic}
\end{minipage} 
\caption{\textit{(a): Space and time accuracy plot of the numerical scheme defined in (\ref{eq_ui}--\ref{eq_imex_bi}) at final time $t = 0.1$. The expression chosen for the second derivative of the entropy function is ${\bf \Phi}''({\bf u}) = 1$, while the Butcher Tableau are defined in SSP-LDIRK3(4,3,3) scheme. (b). Time accuracy of the numerical scheme defined in ~(\ref{eq_ui}--\ref{eq_imex_bi})  at final time $t = 0.01$. The expression chosen for the second derivative of the entropy function is ${\bf \Phi}''({\bf u}) = (u^2(\sin(u) + \sigma) +1)^{-1}$, and $\sigma$ is chosen {such that} {the convexity of ${\bf \Phi}''({\bf u})$ is maintained}, i.e., $\sigma = 2$. In both panels $m = -1.98$ and $ q = 1$ in the expression for the sink/source function in ~\eqref{eq_S_deltas}.}}
\label{fig_accuracy_Du}
\end{figure}

In Fig.~\ref{fig_accuracy_Du} we show the order of accuracy of the numerical scheme, defined in (\ref{eq_ui}--\ref{eq_imex_bi}). To do this, we calculate a reference solution, and then we compute the relative error as follows
\begin{align}
\label{eq:error}
    {\rm error} = \frac{||{\bf D}_{h,\Delta t} - {\bf D}_{\rm ref}||_2}{||{\bf D}_{\rm ref}||_2}
\end{align}
where ${\bf D}_{h,\Delta t}$ is the numerical solution with space step equal to $h$ and time step equal to $\Delta t$, while ${\bf D}_{\rm ref}$ is calculated in two different ways. In panel (a), we consider a simple case, in which the entropy function is defined as $\Phi''(u) = 1$. In this case we have an explicit expression for the exact solution of the system (\ref{eq_D}--\ref{eq_u}), and in order to test the accuracy in space and time, we consider $h = \Delta t \in \{0.04, 0.02, 0.01, 0.005, 0.0025\}$ for ${\bf D}_{h,\Delta t}$.
In panel (b), we fix $h = 10^5$, and we calculate the time accuracy, fixing a reference time step, $\Delta t_{\rm ref} = 10^{-7}$, and a corresponding reference solution ${\bf D}_{\rm ref}$. For ${\bf D}_{h,\Delta t}$ we choose $\Delta t = 0.1\cdot 2^{-k},\, k\in\{2,3,4,5,6,7\}$.
\begin{figure}[H]
	\centering
\begin{minipage}[b]
		{.45\textwidth}
		\centering
	\begin{overpic}[abs,width=\textwidth,unit=1mm,scale=.25]{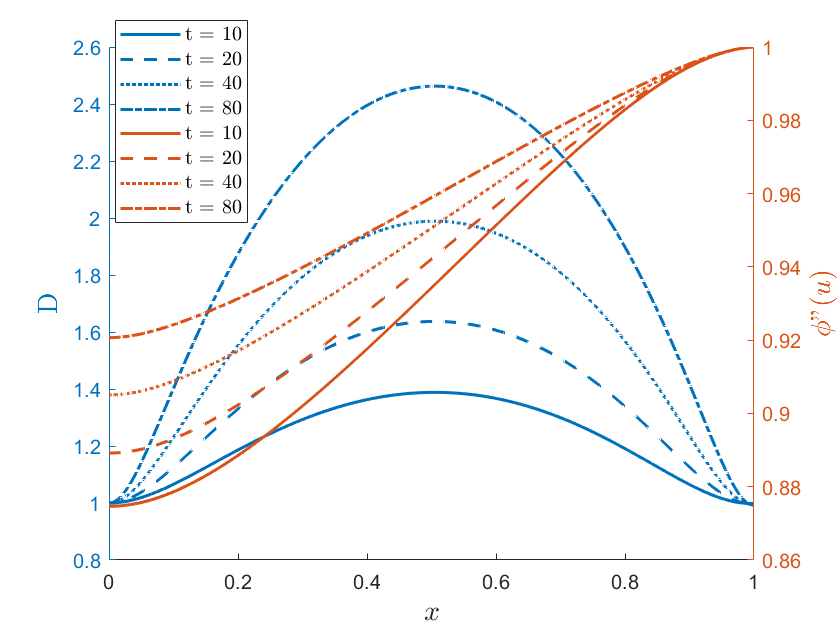}
\put(1,54){(a)}
\end{overpic}
\end{minipage} \\
% 	\begin{minipage}[b]
% 		{.45\textwidth}
% 		\centering
% 	\begin{overpic}[abs,width=\textwidth,unit=1mm,scale=.25]{Figures/D_phi2_time_exponential_2}
%  \put(1,54){(b)}
% \end{overpic}
% \end{minipage}
    \begin{minipage}[b]
		{.45\textwidth}
		\centering
	\begin{overpic}[abs,width=\textwidth,unit=1mm,scale=.25]{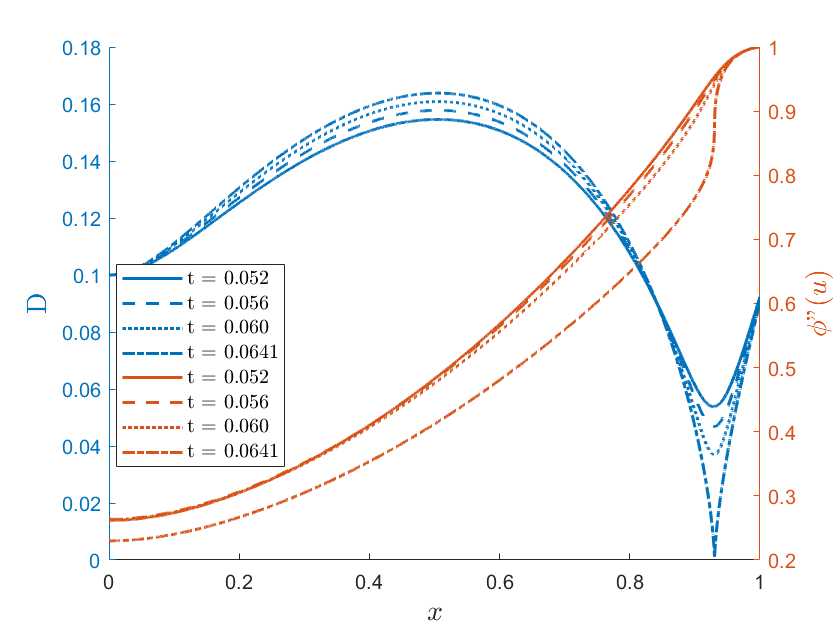}
\put(1,54){(b)}
\end{overpic}
\end{minipage}
	\begin{minipage}[b]
		{.45\textwidth}
		\centering
	\begin{overpic}[abs,width=\textwidth,unit=1mm,scale=.25]{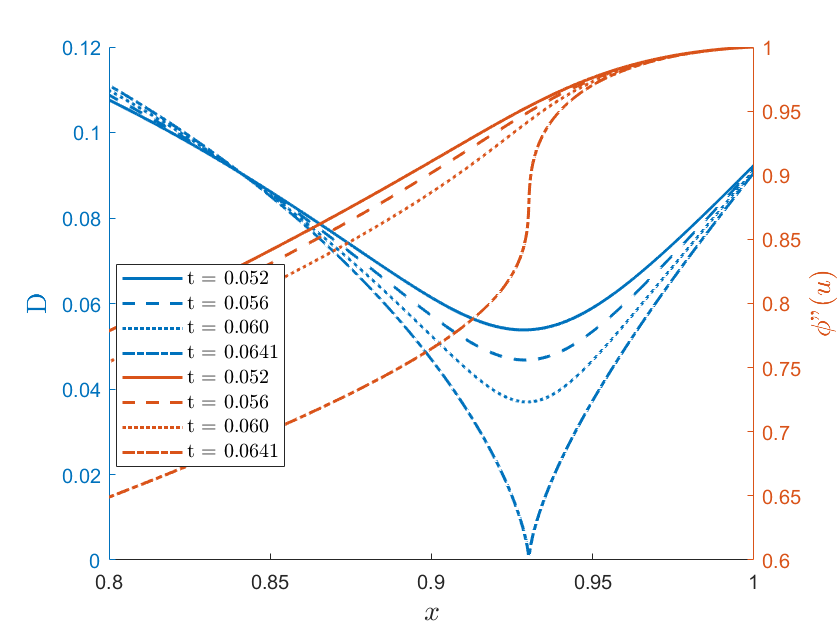}
 \put(1,54){(c)}
\end{overpic}
\end{minipage}
\caption{\textit{Plot of the solution {\bf D} (left part of each plot, in blue), together with the function {${\bf \Phi}''({\bf u})$} (right part of each plot, in orange), for different times. The expression of the second derivative of the entropy function is ${\bf \Phi}''({\bf u}) = \exp(-u)$, the initial condition is {\bf D}(t=0) = 1 (a), {\bf D}(t=0) = 0.1 (b), and a zoom-in of panel (b) in (c), while $m = -1.98$ and $q = 1$ in the expression \eqref{eq:S_delta} and $\Delta t = 10^{-6}, h = 10^{-4}$.}}
\label{fig_exponential_Din1}
\end{figure}

In Fig.~\ref{fig_exponential_Din1} we show the solution ${\bf D}$ (left part of each plot, in blue), together with ${\bf \Phi}''(u)$ (right part of each plot, in orange) with two different initial conditions:  ${\bf D}(t = 0) = 1$ in (a) and ${\bf D}(t = 0) = 0.1$ in (b) and (c) panels. The expression for the second derivative of the entropy function is $\Phi''(u) = \exp(-u)$, while $ m = -1.98$ and $q = 1$ in the expression \eqref{eq_S_deltas}. In these plots we show ${\bf D}$ and ${\bf \Phi}''(u)$ in different moments, to show how they evolve in time. It is interesting to see that after finite time the solution ${\bf D}$ touches zero if the initial condition is small enough. We show ${\bf \Phi}''(u)$ in the same plot because, analyzing equation \eqref{eq:Dt_2}, it is evident that ${\bf \Phi}''(u)$ acts as a weight, dependent on the solution itself ${\bf D}$, for the integral involving $S(x)$. To induce a change in sign in ${\bf D}$, it is necessary for the integral to become negative at a rate that is sufficient to drive the solution towards zero. Consequently, we select $S(x)$ to ensure that ${\bf \Phi}''(u)$ becomes significantly big when $S(x)$ is negative.

Figs.~\ref{fig_fisher_Din1}--\ref{fig_u2sin_Din1} are equivalent to Fig.~\ref{fig_exponential_Din1}, with different expressions for the entropy function. In Fig.~\ref{fig_fisher_Din1} $\Phi''(u) = (u+1)^{-1}$, while in Fig.~\ref{fig_u2sin_Din1} $\Phi''(u) = (u^2(\sin(u)+\sigma)+1)^{-1}$. We observe that, with the right choice of the initial condition, i.e. ${\bf D}(t=0) = 0.1$, the solution touches zero in finite time. Here, again, we choose as final time a few time steps prior the function $\bf D$ touches zero, due to the presence of spurious numerical oscillations arising from the division by $\bf D$.

In Fig.~\ref{fig:minD} we show that it is easier for the numerical scheme to compute the solution $\bf D$ that approaches to zero when we refine the time step. Since there is a division by $\bf D$ in the expression of $\bf u$ (see, e.g., \eqref{eq:u_1}), spurious oscillations start to appear in the numerical simulations when $\bf D$ is close to zero and, for this reason, we are not able to show the exact time when $\bf D$. Fig.~\ref{fig:minD} confirms that, even if we are not able to plot that moment, the trend of the numerical results are in agreement when $\Delta \to 0$. 

\begin{figure}[h]
	\centering
\begin{minipage}[b]
		{.45\textwidth}
		\centering
	\begin{overpic}[abs,width=\textwidth,unit=1mm,scale=.25]{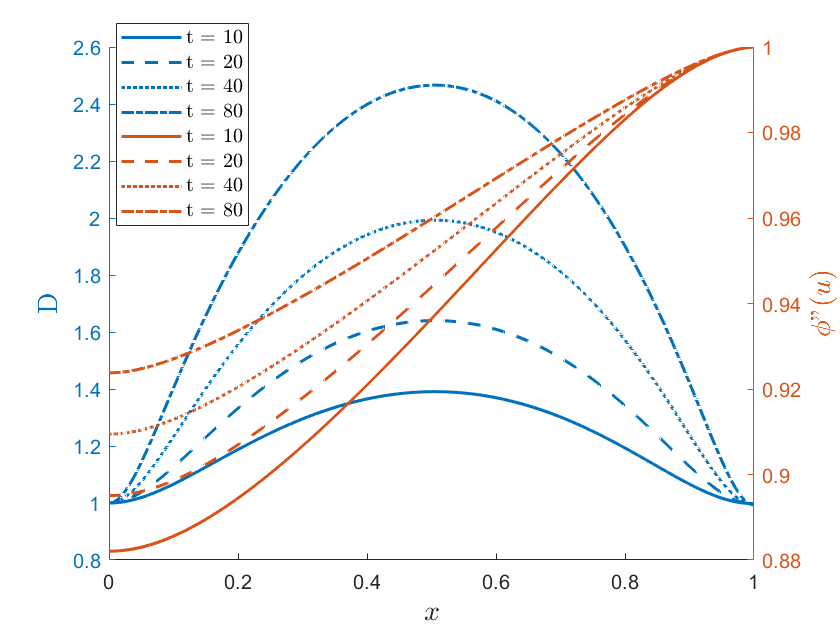}
\put(1,54){(a)}
\end{overpic}
\end{minipage} \\
% 	\begin{minipage}[b]
% 		{.45\textwidth}
% 		\centering
% 	\begin{overpic}[abs,width=\textwidth,unit=1mm,scale=.25]{Figures/D_phi2_time_Fisher_2}
% \put(1,54){(b)}
% \end{overpic}
%\end{minipage}
\begin{minipage}[b]
		{.45\textwidth}
		\centering
	\begin{overpic}[abs,width=\textwidth,unit=1mm,scale=.25]{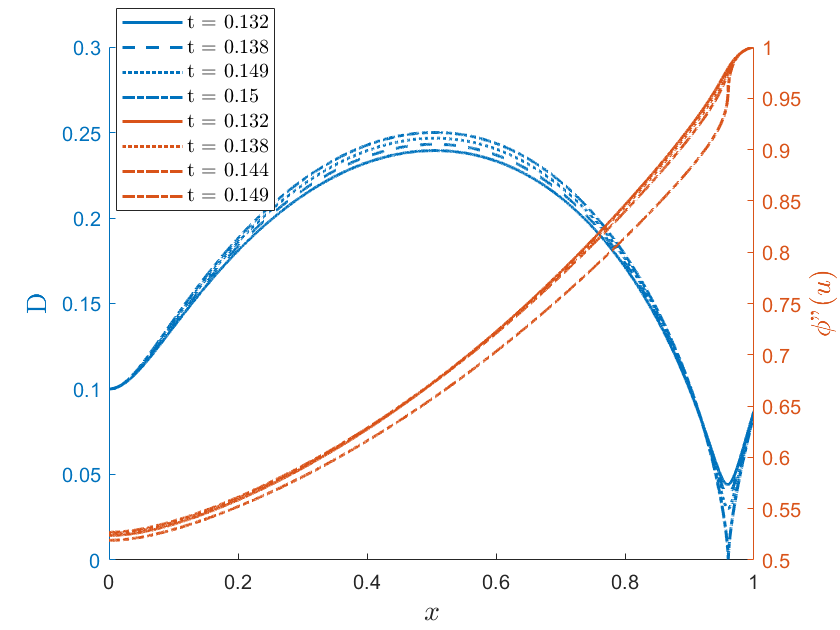}
\put(1,54){(b)}
\end{overpic}
\end{minipage} 
	\begin{minipage}[b]
		{.45\textwidth}
		\centering
	\begin{overpic}[abs,width=\textwidth,unit=1mm,scale=.25]{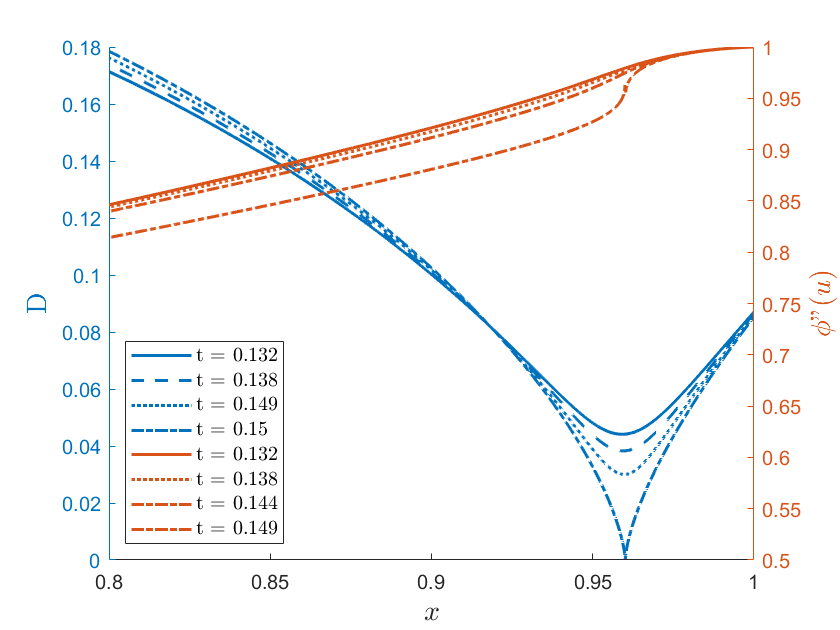}
\put(1,54){(c)}
\end{overpic}
\end{minipage}
\caption{\textit{Plot of the solution {\bf D} (left part of each plot, in blue), together with the function {$\bf \Phi''(u)$} (right part of each plot, in orange), for different times. The second derivative of the entropy function is ${\bf\Phi}''({\bf u}) = (u+1)^{-1}$, and the initial condition is {\bf D}(t=0) = 1 (a), {\bf D}(t=0) = 0.1 (b), and zoom-in (c), while $m = -1.98$ and $q = 1$ in the expression \eqref{eq:S_delta} and $\Delta t = h = 10^{-3}$.}}
\label{fig_fisher_Din1}
\end{figure}

\begin{figure}[h]
	\centering
\begin{minipage}[b]
		{.45\textwidth}
		\centering
	\begin{overpic}[abs,width=\textwidth,unit=1mm,scale=.25]{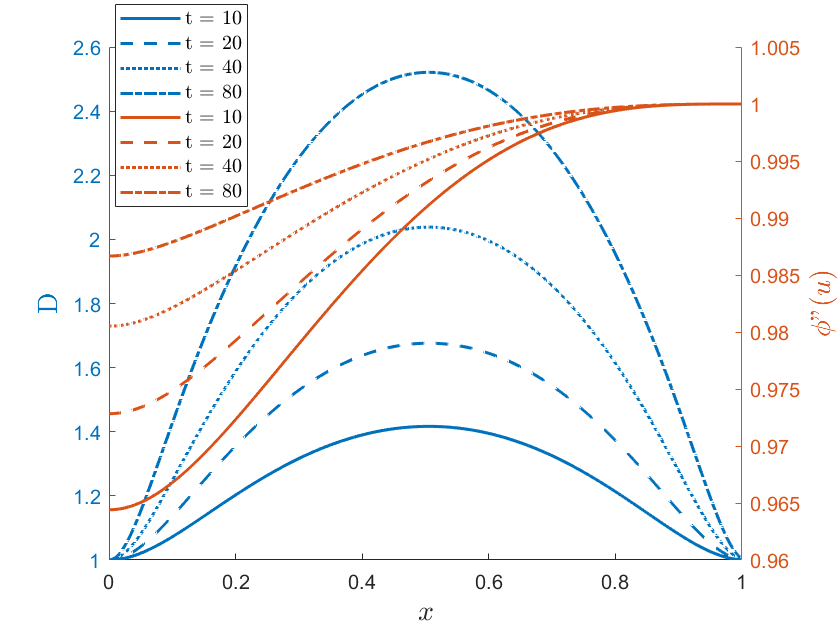}
\put(1,54){(a)}
\end{overpic}
\end{minipage} \\
% \begin{minipage}[b]
% 		{.45\textwidth}
% 		\centering
% 	\begin{overpic}[abs,width=\textwidth,unit=1mm,scale=.25]{Figures/D_phi2_time_u2_sin_2}
% \put(1,54){(b)}
% \end{overpic}
%\end{minipage}
\begin{minipage}[b]
		{.45\textwidth}
		\centering
	\begin{overpic}[abs,width=\textwidth,unit=1mm,scale=.25]{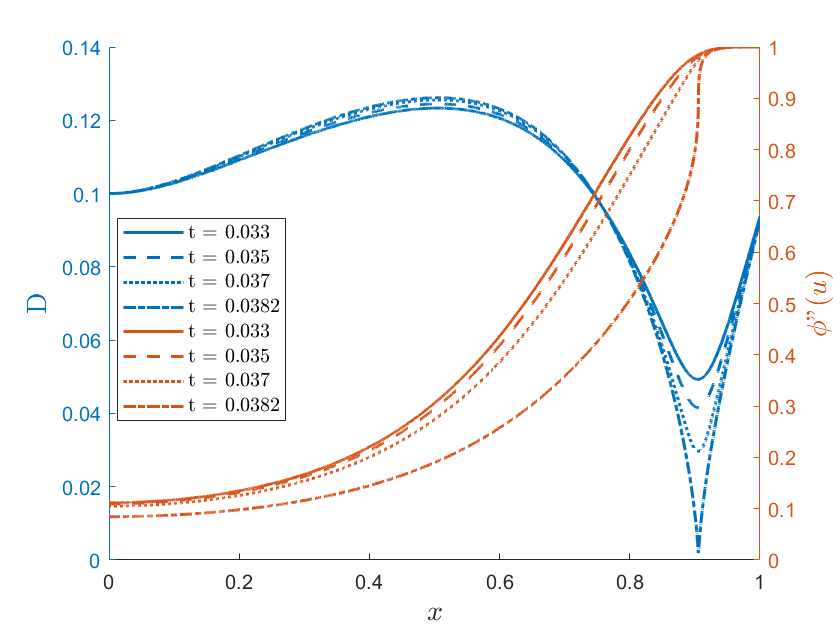}
\put(1,54){(b)}
\end{overpic}
\end{minipage}
\begin{minipage}[b]
		{.45\textwidth}
		\centering
	\begin{overpic}[abs,width=\textwidth,unit=1mm,scale=.25]{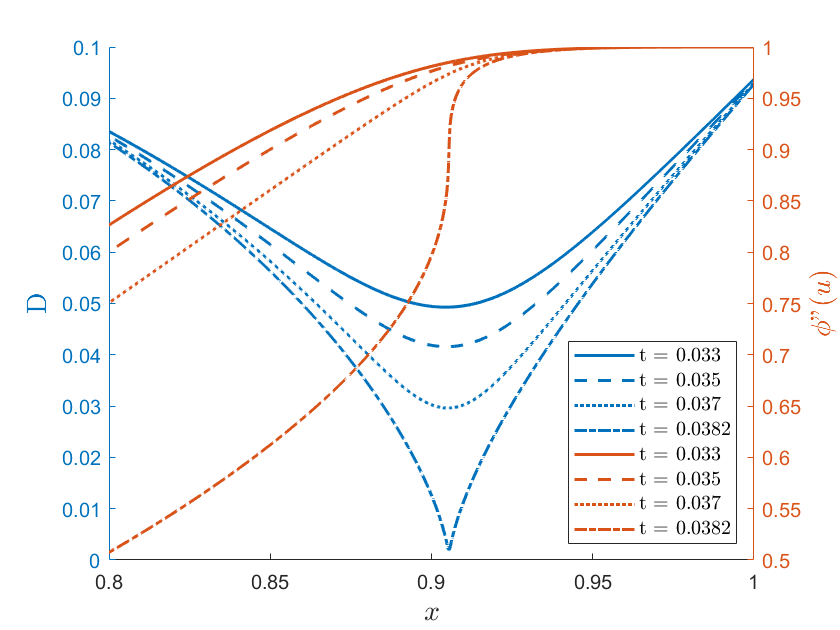}
\put(1,54){(c)}
\end{overpic}
\end{minipage}
\caption{\textit{Plot of the solution {\bf D} (left part of each plot, in blue), together with the function {$\bf \Phi''(u)$} (right part of each plot, in orange), for different times. The second derivative of the entropy function is $\Phi''(u) = (u^2(\sin(u)+\sigma)+1)^{-1}$, and the initial condition is {\bf D}(t=0) = 1 (a), {\bf D}(t=0) = 0.1 (b), and zoom-in (c), while $m = -1.98$ and $q = 1$ in the expression \eqref{eq:S_delta} and $\Delta t = 10^{-6}$ and $h = 10^{-4}$.}}
\label{fig_u2sin_Din1}
\end{figure}
In Fig.~\ref{fig_u2sin_larger_domain} we consider a larger domain, i.e., $x \in [0,1.5]$, to show that the point in which the solution touches zero does not depend neither on the domain, nor on its boundary. On the contrary, it depends on the definition of the function $S(x)$ and on the function $\bf \Phi''(u)$. In Fig.~\ref{fig_u2sin_larger_domain} $m = -1.98/1.5$ and $q = 1$ in \eqref{eq_S_deltas}. 

\begin{figure}[h]
    \centering
    \includegraphics[width=0.65\textwidth,unit=1mm]{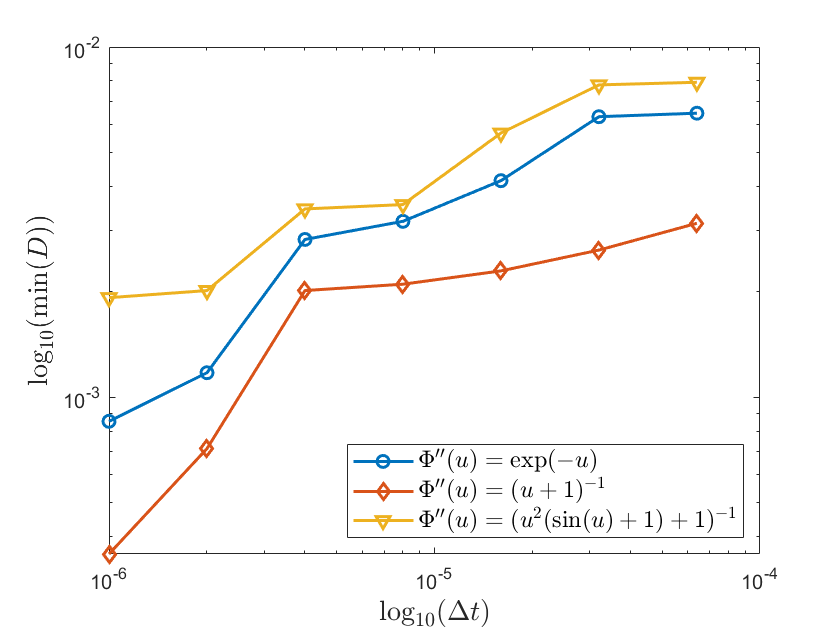}
    \caption{\textit{In this plot we show the minimum of the function $\bf D$ at final time, changing the time step $\Delta t$, for different expressions of ${\bf \Phi}''({\bf u})$. The final time $t_{\rm fin}$ is chosen such that at time $t_{\rm fin} + \Delta t$ spurious oscillations start to appear. }}
    \label{fig:minD}
\end{figure}

\begin{figure}[h]
	\centering
\begin{minipage}[b]
		{.45\textwidth}
		\centering
	\begin{overpic}[abs,width=\textwidth,unit=1mm,scale=.25]{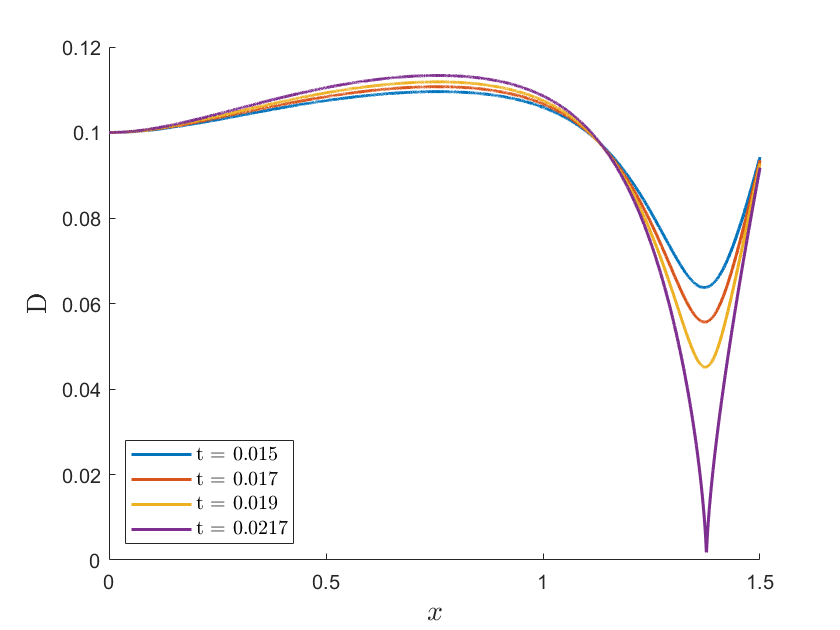}
\put(1,54){(a)}
\end{overpic}
\end{minipage}
\begin{minipage}[b]
		{.45\textwidth}
		\centering
	\begin{overpic}[abs,width=\textwidth,unit=1mm,scale=.25]{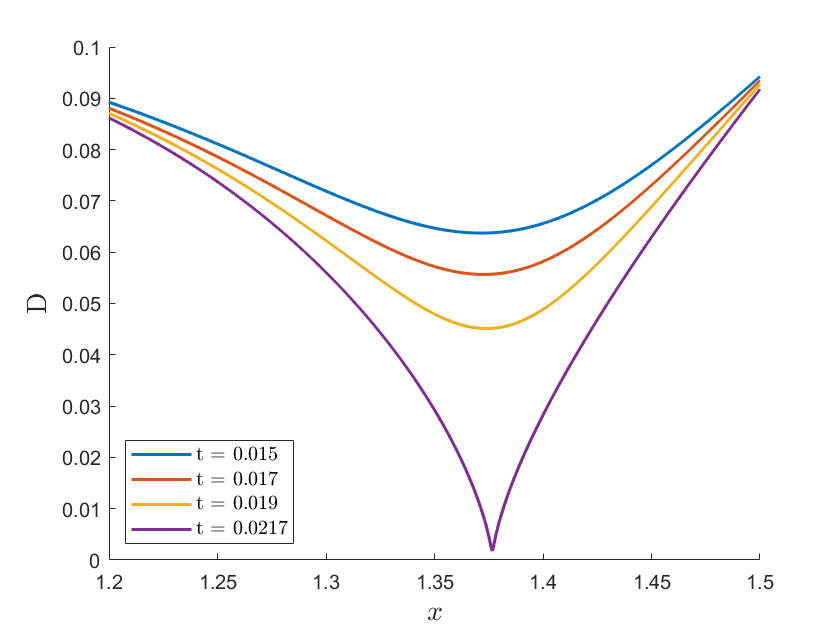}
\put(1,54){(b)}
\end{overpic}
\end{minipage}
\caption{\textit{In this specific example, we consider a larger domain, with $x\in [0,1.5]$. Since the function $R(x) = \int S(x) dx > 0$, we choose another expression for the function $S(x)$ in \eqref{eq_S_deltas}, with $m = -1.98/1.5$ and $q = 1$. Here we show the solution {\bf D}  for different times (a) and a zoom-in (b) in x-direction, to show in which part of the domain the solution is closer to zero. The initial condition is {\bf D}(t=0) = 0.1, $\Delta t = 10^{-6}$ and $h = 10^{-4}$.}}
\label{fig_u2sin_larger_domain}
\end{figure}

\subsection{Time discretization for the $\delta$-system}
In this section we discretize in time the expression in \eqref{eq:Dt_x1}, where the choice for the sink/source function is defined as a sum of delta functions, in \eqref{eq:S_delta}. Here we rewrite \eqref{eq:Dt_x1} as
\begin{align}    
\displaystyle
 \frac{\partial D_1}{\partial t} =& \frac{1}{D_1^2} a^2 \, \Phi'' ( u(x_0,t)) &  x \in [x_0, x_1), \, t\in[0,t_{\rm fin}] \label{eq:Dt_x1_2} \\
 \frac{\partial D_2}{\partial t} = & \frac{1}{D_2^2} (a - b) \left[ a \Phi'' (u(x_0,t)) - b \Phi''(u(x_1, t)) \right] & x \in [x_1, 1], \, t\in[0,t_{\rm fin}]. \label{eq:Dt_x2_2}
\end{align}
where 
\begin{align}
    \label{eq:u_x0_delta_2} u(x_0, t) &= \frac{a(x_1 - x_0)}{D_1(t)} + \frac{(a-b)(1-x_1)}{D_2(t)}, \\
    \label{eq:u_x1_delta_2} u(x_1, t) &= \frac{(a-b) (1-x_1)}{D_2(t)}. 
\end{align}
{To fix ideas, we focus on the case $a > b$.}
The expression for the entropy we are interested in, is the following
\( 
\Phi''(u) = (u^2(\sin(u)+\sigma)+1)^{-1},\)
where $\sigma$ is selected such that
\[\displaystyle \frac{\sigma - 1 }{\sigma + 1} < \frac{b}{a},\]
ensuring that we do not fall into the case \eqref{eq:liminf_delta_1D} addressed by the existing theory.

Equations (\ref{eq:Dt_x1_2}--\ref{eq:Dt_x2_2}) are two non-linear ODE's, and to apply IMEX schemes as before, we multiply and divide by the same quantities the two equations: $D_1$ in the first equation and $D_2$ in the second one, obtaining
\begin{align}    
\displaystyle
 \frac{\partial D_1}{\partial t} =& M_1(D_1,D_2)\, D_1 \qquad  x \in [x_0, x_1), \label{eq:Dt_x1_3} \\
 \frac{\partial D_2}{\partial t} = & M_2(D_1,D_2)\, D_2 \qquad x \in [x_1, 1]. \label{eq:Dt_x2_3}
\end{align}
where
\begin{align}
    M_1(D_1,D_2) & = \frac{1}{D_1^3} a^2 \, \Phi'' ( u(x_0,t)) \\
    M_2(D_1,D_2) &=\frac{1}{D_2^3} (a - b) \left[ a \Phi'' (u(x_0,t)) - b \Phi''(u(x_1, t)) \right]
\end{align}

Now we apply IMEX scheme to (\ref{eq:Dt_x1_2}--\ref{eq:Dt_x2_2}), 
and the scheme reads
\begin{align}
    D_{1,E}^{i}& = D_{1,E}^n + \Delta t\,\sum_{j=1}^{\rm s-1}\widetilde a_{ i,j}M_1(D_{1,E}^{ j},D_{2,E}^{ j})D_{1,I}^{ j}, \quad i = 1,\cdots,s \\
    D_{2,E}^{i}& = D_{2,E}^n + \Delta t\,\sum_{j=1}^{\rm s-1}\widetilde a_{i,j}M_2(D_{1,E}^{ j},D_{2,E}^{ j})D_{2,I}^{ j}, \quad i = 1,\cdots,s \\
    D_{1,I}^{i} &= D_{1,I}^n + \Delta t\,\sum_{j=1}^{\rm s} \widehat a_{ i,j}M_1(D_{1,E}^{ j},D_{2,E}^{ j})D_{1,I}^{ j}, \quad i = 1,\cdots,s \\
    D_{2,I}^{i} &= D_{2,I}^n + \Delta t\,\sum_{j=1}^{\rm s} \widehat a_{ i,j}M_2(D_{1,E}^{ j},D_{2,E}^{ j})D_{2,I}^{ j}, \quad i = 1,\cdots,s. 
\end{align}
The numerical solution is finally updated with
\begin{align}
\label{eq_imex_D1D2_new}
    D_1^{n+1} &= D_{1}^n + \Delta t\sum_{i=1}^{\rm s} \widehat b_{\rm i} M_1(D_{1,E}^i,D_{2,E}^i){D}_{1,I}^i \\
    D_2^{n+1} &= D_{1}^n + \Delta t\sum_{i=1}^{\rm s} \widehat b_{\rm i} M_2(D_{1,E}^{\rm i},D_{2,E}^{\rm i}){D}_{2,I}^{\rm i}.
\end{align}

\subsubsection{Accuracy test and Results}
In this section we first prove the accuracy of the time discretization, where the IMEX scheme used is a third order semi--implicit Runge–Kutta methods is given by the IMEX-SSP3(4,3,3) L-stable Scheme: SSP-LDIRK3(4,3,3).

In Fig.~\ref{fig_accuracy_D1D2} we show the third--order accuracy in time of the numerical scheme in \eqref{eq_imex_D1D2_new}, on the left of the plot (in blue) for the variable $D_1$, and on the right of the plot (in orange) for the variable $D_2$. The initial condition chosen is $D^0_{1,E} = D^0_{2,E} = 1$, the expression for the second derivative of the entropy function is $\Phi''(u) = u^2$ and the other parameters are : $x_0 = 0.1, x_1 = 0.9, a = 1, b = 0.9$. 

In Fig.~\ref{fig_u2sin_Din01_delta} we show the time evolution of the variables $D_1,D_2$ in (a), a zoom-in in time of panel (a) in (b) and of $\Phi''(u(x_1,t))$ in (c), where the final time $t_{\rm fin}$ is chosen such that at time $t_{\rm fin} + \Delta t$ spurious oscillations appear. However, we are able to show the change of sign of the variable $D_2$: at time $t = t_{\rm fin}-\Delta t$ its value is positive, while at time $t = t_{\rm fin}$ it is negative.
 Another agreement is showed in panel (d), where we see the minimum of the variable $D_2$ at final time $t_{\rm fin}$ (chosen in such a way that at time $t_{\rm fin} + \Delta t$ the solution $D_2$ becomes negative). We see that in panel (d) $D_2$ goes closer to zero when we refine the time step, in alignment with the expectations as $\Delta t \to 0$. The initial condition in these tests is $D_1(t=0) = D_2(t=0) = 0.1,$ and the other parameters in (\ref{eq:Dt_x1_2}--\ref{eq:u_x1_delta_2}) are: $x_0 = 0.1, x_1 = 0.9, a = 1, b =0.99$ and $\Delta t = 10^{-6}$. The expression for the second derivative of the entropy function is $\Phi''(u) = (u^2(\sin(u)+\sigma)+1)^{-1},$ with $\sigma = 2$.

\begin{figure}[H]
\centering
\begin{overpic}[abs,width=0.6\textwidth,unit=1mm,scale=.25]{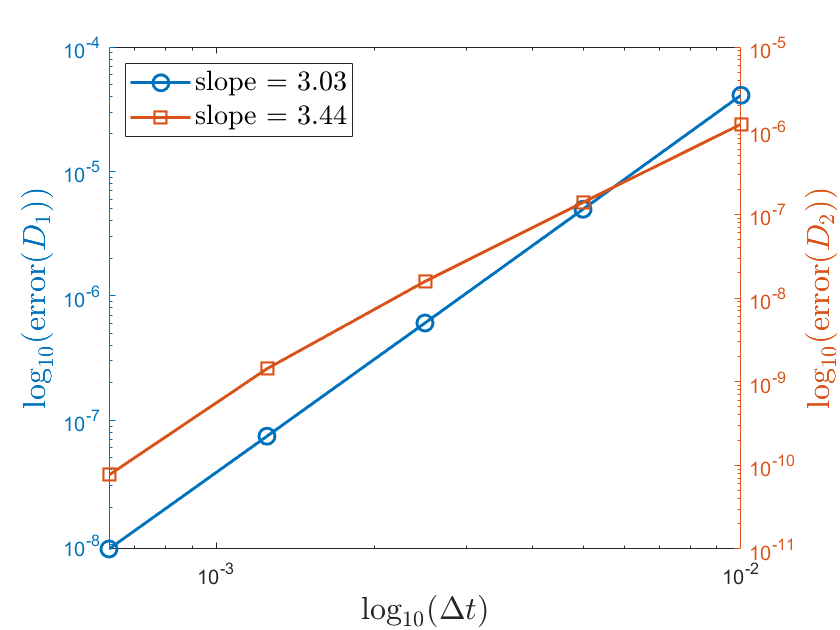}
%\put(1,54){(a)}
\end{overpic}
\caption{\textit{Accuracy plot of the numerical scheme defined in \eqref{eq_imex_D1D2_new}, where the Butcher Tableau are defined in SSP-LDIRK3(4,3,3) scheme. Here we have in the same plot, (on the left part, in blue) the error of the $D_1$ variable, and (on the right part, in orange) the error of the $D_2$ variable. The expression for the error is the same as the one in \eqref{eq:error}, with $\Delta t_{\rm ref} = 10^{-5}$. In this tests the expression for the second derivative of the entropy function is $\Phi''(u) = u^2$. The other parameters are: $x_0 = 0.1, x_1 = 0.9, a = 1, b = 0.9$.}}
\label{fig_accuracy_D1D2}
\end{figure}

\begin{figure}[H]
\begin{minipage}[b]
		{.45\textwidth}
		\centering
\begin{overpic}[abs,width=\textwidth,unit=1mm,scale=.25]{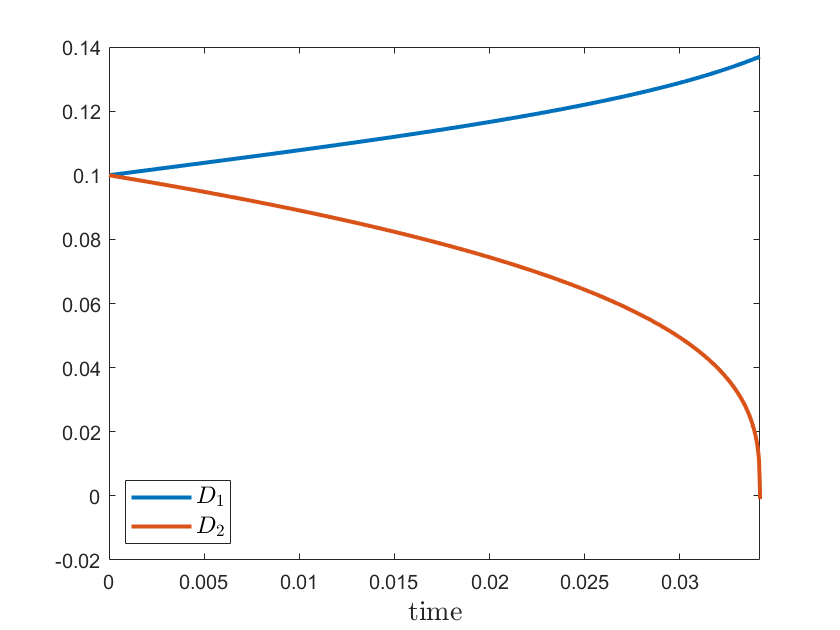}
\put(1,54){(a)}
\end{overpic}
\end{minipage}
\begin{minipage}[b]
		{.45\textwidth}
		\centering
\begin{overpic}[abs,width=\textwidth,unit=1mm,scale=.25]{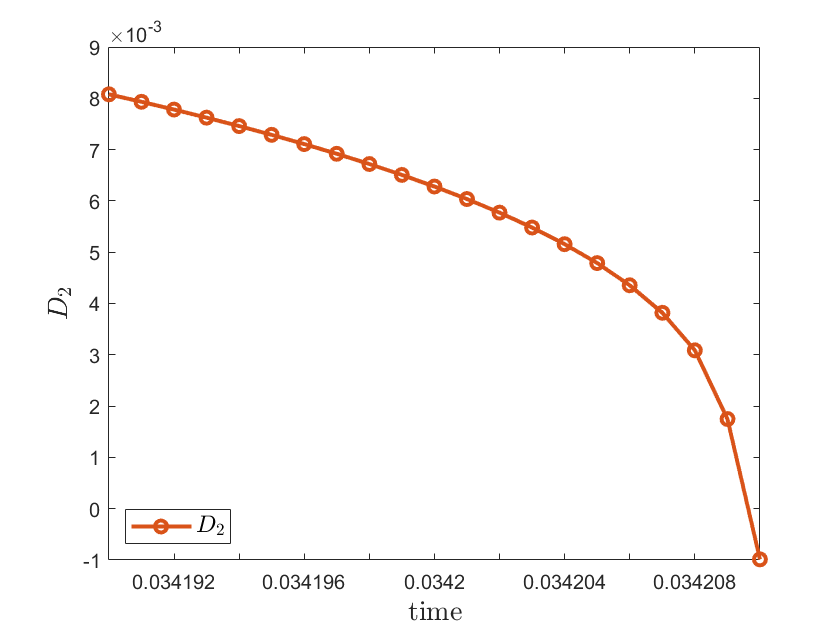}
\put(1,54){(b)}
\end{overpic}
	\end{minipage}
	\begin{minipage}[b]
		{.45\textwidth}
		\centering
	\begin{overpic}[abs,width=\textwidth,unit=1mm,scale=.25]{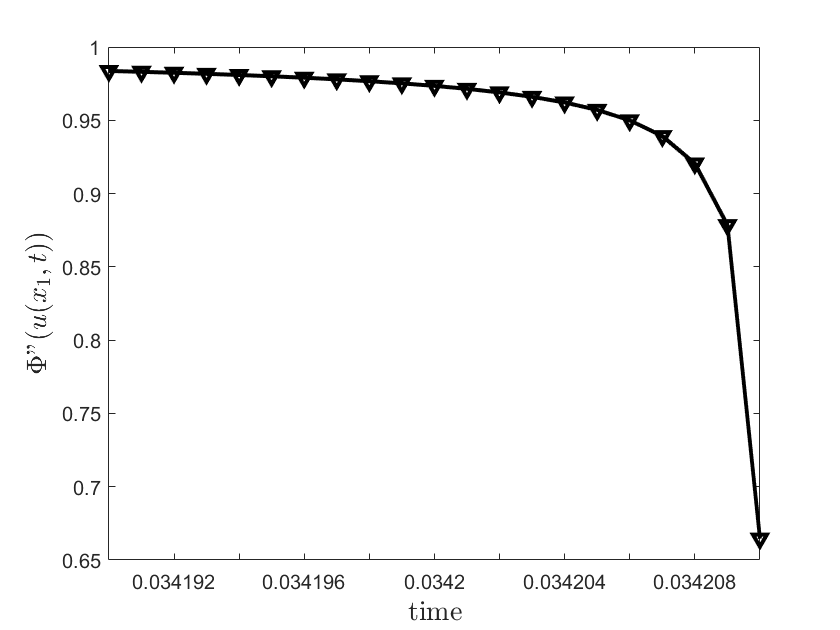}
\put(1,54){(c)}
\end{overpic}
\end{minipage}
\begin{minipage}[b]
		{.45\textwidth}
		\centering
	\begin{overpic}[abs,width=\textwidth,unit=1mm,scale=.25]{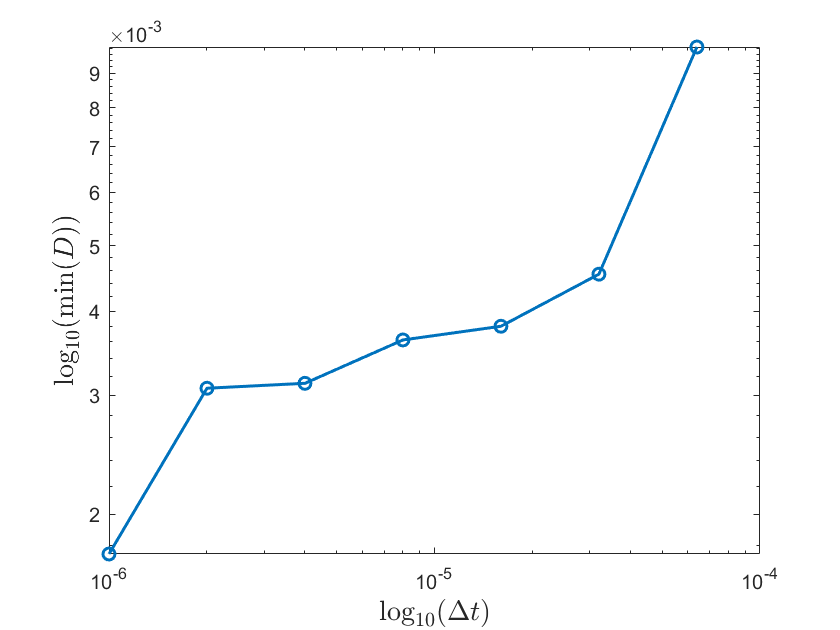}
\put(1,54){(d)}
\end{overpic}
\end{minipage}
\caption{\textit{Time evolution of the functions $D_1,D_2$  in (a), zoom of $D_2$ in (b) and of $\Phi''(u(x_1,t))$ in (c). Initial conditions ${D_1}(t = 0) = D_2(t = 0) = 0.1$ and the final time $t_{\rm fin}$ is chosen such that at time $t_{\rm fin} + \Delta t$  spurious oscillations start to appear. In this test the expression for the second derivative of the entropy function is $\Phi''(u) = (u^2(\sin(u)+2)+1)^{-1}$, with $\sigma = 2$. The other parameters are: $x_0 = 0.1, x_1 = 0.9, a = 1, b = 0.99$ and $\Delta t = 10^{-6}$. In panel (d), by varying the value of the time step, we display the minimum of $D_2$ at the final time $t_{\rm fin}$, such that at time $t_{\rm fin} + \Delta t$ the solution $D_2$ becomes negative.}}
\label{fig_u2sin_Din01_delta}
\end{figure}

\section{Conclusions}
In this paper we provided a proof of the well-posedness of self-regulating processes that model biological transportation networks, as presented in \cite{portaro2023}. Such models arise from a constrained minimization problem of an entropy dissipation coupled to the conservation law for a quantity representing the concentration of a chemical species, ions, nutrients or material pressure. Our focus was particularly directed toward the 1D setting, where we specifically investigate Dirichlet and Neumann boundary conditions. The analysis leads to either a local existence and uniqueness result if $D$ reaches zero in finite time or, alternatively, a global one. We explored several conditions under which global existence is guaranteed.

Moreover, to gain a deeper understanding of the behavior of $D$ and its influence on the system under investigation, we systematically studied a range of possible scenarios. This entails analyzing the system with a signed measure distribution of sources and sinks.
%which poses questions on whether $D$ approaches zero or not.

Finally, we performed several numerical tests in which the solution $D$ approaches zero. These tests validate and provide further evidence of the local existence in specific cases, as hinted at in previous considerations.

As implied by the title, a natural extension of this work would be the development of a theory in 2D, that is already under investigation.

\end{document}